\documentclass[runningheads]{llncs}
\usepackage[T1]{fontenc}
\usepackage{graphicx}
\usepackage{hyperref}
\usepackage{color}

\usepackage{enumerate}
\usepackage[shortlabels]{enumitem}
\usepackage{verbatim}
\usepackage{booktabs}       
\usepackage{multirow}
\usepackage{url}            
\usepackage[numbers]{natbib}

\usepackage{geometry}
\usepackage[short]{optidef}
\usepackage{algorithm}
\usepackage{thmtools}

\usepackage[utf8]{inputenc}
\usepackage{algorithm}
\usepackage{algorithmic}

\usepackage[utf8]{inputenc} 
\usepackage[T1]{fontenc}    
\usepackage[english]{babel}

\usepackage{amsfonts}
\usepackage{nicefrac}       
\usepackage{microtype}      
\usepackage{lmodern}
\usepackage{amssymb,amsmath}
\usepackage{relsize}
\usepackage{array}
\usepackage{stmaryrd}
\usepackage{bbm,bm}
\usepackage{subfig}
\usepackage{latexsym}
\usepackage{xcolor}
\usepackage{graphicx}
\usepackage{booktabs}
\usepackage{url}
\usepackage{marvosym}
\declaretheorem{Fact}
\DeclareMathOperator{\argmin}{argmin}

\definecolor{gold}{rgb}{0.85,0.65,0}

\numberwithin{subsection}{section}
\numberwithin{theorem}{section} 
\numberwithin{lemma}{section} 
\numberwithin{definition}{section} 
\numberwithin{proposition}{section} 
\numberwithin{lemma}{section} 
\numberwithin{corollary}{section} 
\numberwithin{Fact}{section} 

\numberwithin{equation}{section}

\let\emptyset\varnothing

\def\R{{\mathbb{R}}}

\def\cC{{\cal C}}
\def\cD{{\cal D}}

\def\cJ{{\cal J}}

\def\cO{{\cal O}}

\def\cU{{\cal U}}

\def\1{{\bm 1}}

\newcommand{\dist}{\operatorname{dist}}

\DeclareMathOperator*{\argmax}{arg\,max}

\DeclareMathOperator{\diag}{diag}

\DeclareMathOperator{\conv}{conv}

\begin{document}
\title{ \textbf{Computing Competitive Equilibrium for Chores: \\ Linear Convergence and Lightweight Iteration} }
\titlerunning{Computing CE for Chores}
%
\author{He Chen\inst{1} \and
Chonghe Jiang\inst{2} \and
Anthony Man-Cho So\textsuperscript{(\Letter)}\inst{3} }
\authorrunning{H. Chen et al.}
%
\institute{The Chinese University of Hong Kong, Hong Kong SAR, China \\
\email{hechen@link.cuhk.edu.hk}\\ \and 
The Chinese University of Hong Kong, Hong Kong SAR, China \\
\email{chjiang@link.cuhk.edu.hk}  \\ \and
The Chinese University of Hong Kong, Hong Kong SAR, China \\
\email{manchoso@se.cuhk.edu.hk}
}
\maketitle              
\begin{abstract}
 Competitive equilibrium (CE) for chores has recently attracted significant attention, with many algorithms proposed to approximately compute it. However, existing algorithms either lack iterate convergence guarantees to an exact CE or require solving high-dimensional linear or quadratic programming subproblems. This paper overcomes these issues by proposing a novel unconstrained difference-of-convex formulation, whose stationary points correspond precisely to the CE for chores. We show that the new formulation possesses the local error bound property and the Kurdyka-{\L}ojasiewicz property with an exponent of $1/2$. Consequently, we present the first algorithm whose iterates provably converge linearly to an exact CE for chores. Furthermore, by exploiting the max structure within our formulation and applying smoothing techniques, we develop a subproblem-free algorithm that finds an approximate CE in polynomial time. Numerical experiments demonstrate that the proposed algorithms outperform the state-of-the-art method.

\end{abstract}
\section{Introduction}
The competitive market equilibrium problem, which dates back to the nineteenth century, was initially proposed by Leon Walras \cite{walras2014leon,orlin2010improved}. This problem involves a market allocating $m$ items among $n$ agents with the objective of achieving both efficiency and fairness. Walras gave some principles of an ideal status for such a market but did not discuss its existence. In the seminal work of Arrow and Debreu \cite{arrow1954existence}, they proved that under the concavity of the utility functions of agents, there exists a group of allocations and prices meeting Walras's requirements, which is called \textit{competitive equilibrium} (CE) and satisfies (i) for each agent, the total price/payments of collected items equals his budget/expected amount; (ii) every agent gets his optimal bundle; (iii) each item is allocated.

In the classical Fisher market \cite{brainard2005compute}, which is a special case of Walras's model, the items are goods, and agents spend money to buy them. In this context, computing CE is equivalent to finding the Karush-Kuhn-Tucker (KKT) points of certain convex programs, such as Eisenberg-Gale program \cite{eisenberg1959consensus,gale1989theory} and Shmyrev’s program \cite{shmyrev2009algorithm,birnbaum2011distributed}. Various algorithms have been developed to leverage the specific structures of these convex programs to compute CE efficiently, particularly for linear utilities. Examples include the interior point method \cite{jain2007polynomial,ye2008path}, combinatorial methods \cite{chaudhury2018combinatorial,devanur2008market,duan2015combinatorial}, and first-order methods \cite{birnbaum2011distributed,gao2020first,nan2024convergence}. 

In the early 2010s, Budish \cite{budish2011combinatorial} considered another interesting case, where the items are chores and agents are paid for doing them; for example dividing job shifts among workers, papers among reviewers, and teaching load among faculty. Budish's work has sparked intense research on the properties of CE for chores in the past ten years. As it turns out, computing CE for chores is much more challenging than for goods. First, CE for chores cannot be captured by a convex program since the set of CE can be nonconvex and disconnected  \cite{bogomolnaia2017competitive,bogomolnaia2019dividing}. In contrast to the convex programs in the goods setting, Bogomolnaia et al. \cite{bogomolnaia2017competitive} proposed a nonconvex Eisenberg-Gale-like program whose KKT points correspond to CE for chores. Second, the nonconvex program for CE for chores has implicit open constraints to ensure each agent is allocated chores. To compute a CE for chores, one needs to find a KKT point that lies in the interior of the domain. Unfortunately, it is shown by Bogomolnaia et al. \cite{bogomolnaia2017competitive} that first-order methods typically converge to the boundary, leading to undesirable allocations. Finally,  Chaudhury and Mehlhorn \cite{chaudhury2018combinatorial} showed the PPAD-hardness of computing a CE for chores in the exchange model, which is in sharp contrast to the strong polynomial time algorithms in
the goods setting \cite{garg2019strongly}.  

Despite these difficulties, there is a growing interest in designing algorithms that compute an approximate CE for chores in polynomial time. 
 Chaudhury and Mehlhorn \cite{chaudhury2018combinatorial} proposed a combinatorial algorithm to compute an $\epsilon$-approximate CE with equal income (CEEI) in $\tilde{\cO}(nm/\epsilon^2)$ iterations with time complexity of $\tilde{\cO}(n^4m^2/\epsilon^2)$. However, this algorithm faces some numerical issues, i.e., the bit-length of the prices and the allocation is conjectured to grow exponentially with the iterations \cite{duan2015combinatorial}, which undermines its practicality \cite{chaudhury2024competitive}. Later, Boodaghians et al. \cite{boodaghians2022polynomial} designed a practical exterior-point method (EPM), which addressed the open constraint issue by keeping iterates away from the boundary. To find an $\epsilon$-CE, the EPM theoretically requires to solve $\tilde{\cO}(n^3/\epsilon^2)$ QPs.  More recently,  Chaudhury et al. \cite{chaudhury2024competitive} made remarkable progress. They developed a novel nonlinear program that eliminates implicit open constraints, with each KKT point corresponding to a CE for chores. A greedy Frank-Wolfe (GFW) method was used to find an $\epsilon$-CE in $\tilde{\cO}(n/\epsilon^2)$ iterations. Taking advantage of the new formulation and efficient solvers for LP, the GFW method outperforms previous algorithms both theoretically and empirically. 

Though the EPM and GFW methods perform well in practice and have well-established complexity bounds to compute an approximate CE for chores, they face two significant limitations: \textbf{(L1)} As noted by Chaudhury et al. \cite{chaudhury2024competitive}, even when $\epsilon$ is very small, an $\epsilon$-CE can still be far from an exact CE, making it unsatisfactory in some cases. However, none of the existing algorithms ensure iterates " move to a 'nearby' exact CE" in the chores setting. This contrasts with the classical goods setting, where the distance between iterates and CE can be shown to converge linearly to zero (see, e.g., \cite[Theorem 2.1]{gao2020first}). \textbf{(L2)} Both the EPM and GFW methods require solving high-dimensional QPs or LPs as subproblems. Since off-the-shelf solvers usually adopt interior point methods, solving a subproblem requires up to $\cO((mn)^{3.5})$ time \cite{vavasis1996primal}. This limits the performance of these methods in large-scale settings.

\paragraph{Our Contributions.} In this paper, we consider the problem of computing CE for chores and design two algorithms to resolve the above limitations. To begin, invoking the formulation developed in \cite{chaudhury2024competitive}, we present a novel unconstrained \textit{difference-of-convex} (DC) formulation, whose stationary points have a one-to-one correspondence to CE for chores.  One of our main contributions is to prove that this DC problem possesses the Kurdyka-{\L}ojasiewicz (K{\L}) property with an exponent of $1/2$. This result is nontrivial as it is rooted in the special structure of the DC formulation and places the problem of computing CE for chores under the classical framework of minimizing K{\L} functions with exponent $1/2$. More importantly, leveraging this property, we show for the first time that the difference-of-convex algorithm (DCA) provably achieves iterate convergence to a stationary point, i.e., an exact CE for chores, at a linear rate, thereby addressing limitation {(L1)}. In each iteration, the DCA subproblem is a convex QP with simplex constraints, which can be solved by first-order methods (e.g., projected gradient descent (PGD) or mirror descent) efficiently.

Towards the goal of light computation and addressing {(L2)}, we further design a \textit{smoothed gradient method with rounding} (SGR) based on the DC formulation. By smoothing the $\max$ terms of the objective function via entropy regularization, we derive a nonconvex smooth approximation problem for the DC formulation. The SGR employs a simple gradient descent method with a novel rounding procedure for this approximation problem,  yielding an $\epsilon$-CE in $\tilde{\cO}({m^2}/{\epsilon^3})$ iterations. While the iteration complexity is higher than existing methods, the SGR only requires computing a gradient and rounding the iterate in each iteration, which can be done in $\cO(nm)$ and $\cO(m^2)$ arithmetic operations, respectively. Compared with the EPM and GFW methods that solve a QP or LP with $mn$ variables or constraints, our SGR involves significantly cheaper computation per iteration, making it more practical for large-scale settings.  

We remark that both SGR and DCA crucially rely on the DC formulation. They adopt different approximations for the objective function, leading to their distinct advantages.\footnote{One may ask whether we can combine the advantages of DCA and SGR to give an algorithm addressing both (L1) and (L2). Intuitively, this is difficult since algorithms that converge linearly to stationary points, such as the DCA, PGD, and proximal point method, usually need to solve subproblems. Still, this problem is interesting and we leave it for future work.} Specifically, the DCA majorizes the concave smooth part of the objective function via linearization while the SGR smoothes the convex nonsmooth part. Due to their high-level similarity, the rounding procedure of SGR can be applied to the DCA to derive an iteration complexity bound of $\tilde{\cO}(m/\epsilon^2)$ for finding an $\epsilon$-CE (comparable to the rate of GFW), without compromising the iterate convergence property.

Finally, we demonstrate the practicality and efficiency of our algorithms via numerical experiments. We generate the data from various distributions and evaluate the CPU time needed by different algorithms to find an approximate CE for chores. Numerical results show that when agents and chores are equal in number, i.e., $m=n$, the CPU time of DCA and SGR is significantly lower than that of GFW. Moreover, the CPU time of SGR increases very slightly compared with the GFW method as $n$ grows, which is consistent with our theory. For the case where $n\gg m$, the DCA (resp. SGR) can be $7$ (resp. $50$) times faster than the GFW method. The competitive numerical performance of our algorithms, as well as their theoretical advantages, demonstrate the power of our approach.

\paragraph{Further Related Work.} The market equilibrium problem has many variants addressing different settings. These variants include the auction markets \cite{conitzer2022pacing}, stochastic online Fisher markets \cite{jalota2022stochastic,balseiro2021regularized}, and Fisher markets with extra constraints meeting fairness or capacity requirements \cite{jalota2023fisher,peysakhovich2023implementing}. Generally, these equilibrium problems can be formulated as convex programs and solved via convex optimization methods. The techniques used in this paper, e.g., the smoothing strategy, error bound, and K{\L} property, may also find applications to these problems.

We then review the technical tools used in this paper. We begin with the smoothing method, which was initially proposed by Nesterov \cite{nesterov2005smooth}. The main idea is to find a smooth approximation of the structured nonsmooth terms of the objective function in convex optimization. Consequently, instead of dealing with a nonsmooth problem, one only needs to solve a smooth convex problem that can be efficiently solved by the gradient descent method. This idea is particularly useful in convex optimization; see, e.g., \cite{yang2022auc,chambolle2011first,lin2020near,zhao2023primal} and references therein for its various applications. Moreover, this technique can be generalized to nonconvex optimization; see the recent progress of Zhao \cite{zhao2023primal}. However, the smoothing method does not offer any iterate convergence guarantee. To derive a convergence guarantee for iterative algorithms, a prototypical approach is to develop a suitable error bound condition for the targeted problem, which bounds the distance of vectors
in a test set to a given set by a residual function.  Luo and Tseng \cite{luo1993error} gave an example of the error bound (Luo-Tseng error bound) and developed the linear convergence of a host of iterative methods based on it. Later, Zhou and So \cite{zhou2017unified} presented a unified approach to obtaining error bounds for structured convex problems, giving iterate convergence results for many modern applications. Another popular approach to establishing iterate convergence is through the K{\L} property \cite{bolte2010characterizations,attouch2013convergence}, which stipulates that the growth of a function is locally bounded by the norm of the function's gradient with a certain exponent. The K{\L} property is closely related to the error bound property. In particular, the K{\L} property with exponent $1/2$ can be implied by Luo-Tseng error bound property \cite{li2018calculus} in the convex setting. A typical approach to obtaining linear convergence results for various iterative methods is to first establish an error bound and then prove the K{\L} property with exponent $1/2$ for the problem at hand \cite{liu2019quadratic,wang2023linear}. It is worth noting that error bounds have been proven for some convex formulations of the Fisher market \cite{nan2024convergence,gao2020first}. In comparison, the error bound we establish is for a nonconvex formulation of the chores market and requires far more effort in analyzing the local property.

\paragraph{Organization.} This paper is organized as follows. Sec. \ref{sec:CE} reviews the formal definition of CE (for chores). Sec. \ref{sec:formulation} presents a new DC formulation for determining CE and establishes its local error bound property and K{\L} property with exponent $1/2$. In Sec. \ref{sec:dca}, we propose a DCA to compute CE, whose iterates converge linearly to a CE. In Sec. \ref{sec:sgr}, we propose a smoothed gradient method with a rounding procedure and develop its non-asymptotic rate for computing an approximate CE. In Sec. \ref{sec:numerical}, we report numerical results to show the superior performance of our algorithms in practice. Finally, we give some closing remarks in Sec. \ref{sec:conclusion}.

\paragraph{Notation.} The notation used in this paper is mostly standard. We use $[m]$ to denote the integer set $\{1,2,\ldots,m\}$ and $\|\cdot\|$ to represent the standard Euclidean norm. We use $e_j$ to denote the $j$-th standard basis vector in $\R^m$. Let $\mu\in\R^m$ be a vector. The distance from the point $\mu$ to a set $\cD\subseteq\R^m$ is defined by $\dist(\mu,\cD)\coloneqq \min_{y\in\cD}\ \|\mu-y\|$, and the projection of $\mu\in\R^m$ onto $\cD\subseteq\R^m$ is defined by $\Pi_{\cD}(\mu)\coloneqq \argmin_{y\in\cD}\ \|\mu-y\|$. The sum of sets $\cC,\cD\subseteq\R^m$ is defined by $\cC+\cD\coloneqq\{x+y:x\in\cC,y\in\cD\}$.

\section{Preliminary: Definition of CE for Chores}\label{sec:CE}
As previously discussed, a CE satisfies three fundamental conditions on agents' expected returns, optimal allocations, and item clearance. We now provide a formal depiction of CE within the context of chores.

Consider allocating $m$ divisible chores among $n$ agents, where agent $i$ has a disutility of $d_{ij}$ ($d_{ij}>0$) for chore $j$ per unit. To facilitate this allocation, agents engaging in each chore $j$ are paid at a price of $p_j$ per unit. Each agent $i$ seeks to earn an expected amount $B_i$ ($B_i>0$) from participating in these chores while simultaneously minimizing their total disutility. Let $x_i\in\R^{m}_+$ denote the allocation vector of agent $i$ and $x_{ij}$ denote the $j$-th component of $x_i$, i.e., the amount of chore $j$ allocated to agent $i$. CE for chores is defined as follows; see, e.g., \cite[Definition 1]{chaudhury2024competitive}.
\begin{definition}[Competitive Equilibrium for Chores]\label{def:CE}
We say that a price $p\in\R^m_+$ and an allocation $x\in\R^{n\times m}_+$ satisfy competitive equilibrium (CE) if and only if
\begin{enumerate}[label={{\rm (E\arabic*).}}, leftmargin=*]
    \item $p^{\top}x_i=B_i$ for all $i\in[n]$; 
    \item $d_i^{\top}x_i\leq d_i^{\top}y_i$ for all $y_i\in\R^m_+$ such that $p^{\top}y_i\geq p^{\top}x_i$, for all $i\in[n]$;
    \item $\sum_{i\in[n]}x_{ij}=1$ for all $j\in[m]$.
\end{enumerate}
\end{definition}
Condition (E2) says that agent $i$ minimizes his disutility under his expected amount. This is equivalent to that agent $i$ only chooses chores from the set $\argmax_j\{\frac{p_j}{d_{ij}}\}$. 
\begin{Fact}\label{le:e2}
    Condition {\rm (E2)} is equivalent to: For all $i\in[n]$,
    \begin{equation}\label{eq:e2}
        \frac{p_j}{d_{ij}}\geq \frac{p_{j^{\prime}}}{d_{ij^{\prime}}}\text{ for all } j\in\{j:x_{ij}>0\}\text{ and } j^{\prime}\in[m].
    \end{equation}
\end{Fact}
\begin{proof}
    The proof is similar to that of the case of the goods (see, e.g., \cite[Sec. 3.1]{jalota2022stochastic}), which is based on the individual optimization problem. We omit it for brevity.
\end{proof}
Fact \ref{le:e2} ensures that $p>0$ for all CE $(p,x)$. To see this, notice that for each $j\in[m]$, there exists an $i$ such that $x_{ij}>0$ by (E3). Then $p_j>0,j\in[m]$ follows from \eqref{eq:e2} and $p\in\R^m_+\setminus\{0\}$. 
\begin{corollary}\label{co:p}
    If $(p,x)$ is a CE, then $p_j>0$ for $j\in[m]$.
\end{corollary}
Next, we give the definition of approximate CE (see, e.g., \cite[Definition 2.1]{boodaghians2022polynomial}).
\begin{definition}[Approximate CE for Chores]\label{def:aCE}
We say that a price $p\in\R^m_+$ and an allocation $x\in\R^{n\times m}_+$ satisfy $\epsilon$-CE if and only if
\begin{enumerate}[label={{\rm (A\arabic*).}}, leftmargin=*]
    \item $(1-\epsilon)B_i\leq p^{\top}x_i\leq \frac1{1-\epsilon}B_i$ for all $i\in[n]$; 
    \item $(1-\epsilon)d_i^{\top}x_i\leq d_i^{\top}y_i$ for all $y_i\in\R^m_+$ such that $p^{\top}y_i\geq p^{\top}x_i$, for all $i\in[n]$;
    \item $1-\epsilon\leq \sum_{i\in[n]}x_{ij}\leq\frac1{1-\epsilon}$ for all $j\in[m]$.
\end{enumerate}
\end{definition}
\section{Difference-of-convex Formulation}\label{sec:formulation}
In this section, we study the computation of CE for chores and propose our formulation.
We begin with the chores dual redundant formulation introduced by Chaudhury et al. \cite{chaudhury2024competitive}:
\begin{equation}\label{eq:redundant}
\begin{array}{rl}
    \max\limits_{\beta\in\R^n_+,p\in\R^m_+}\ & \quad \mathlarger{\sum}\limits_{j\in[m]}p_j-\mathlarger{\sum}\limits_{i\in[n]}B_i\log(\beta_i), \\
   {\rm subject\ } {\rm to} &\quad p_j\leq \beta_id_{ij} \qquad \forall~ i\in[n],j\in[m],\\
    & \quad\mathlarger{\sum}\limits_{j\in[m]}p_j=\mathlarger{\sum}\limits_{i\in[n]}B_i.
\end{array}\tag{Chores Dual Redundant}
\end{equation}
Here, the constraint ${\sum}_{j\in[m]}p_j={\sum}_{i\in[n]}B_i$ is a direct corollary of (E1) and (E2) of Definition \ref{def:CE}. By incorporating this redundant constraint in the program, the variables $\beta_i, i\in[n]$ are lower bounded by ${\sum}_{i\in[n]}B_i/(m\max_{j}d_{ij})$. Consequently, the objective value cannot tend towards infinity within the feasible region, and the implicit open constraints $\beta_i>0, i \in [n]$ are eliminated. 

We consider simplifying this formulation by extracting the constraints.
Observe that the objective function is monotonically decreasing to $\beta_i$ and $\sum_{j\in[m]}p_j$ in the objective is redundant. We can translate \eqref{eq:redundant} into a minimization problem that merely concerns $p$.
\begin{equation*}
   \begin{array}{rl}
    \min\limits_{p\in\R^m_+}\ & \quad \mathlarger{\sum}\limits_{i\in[n]}B_i\log\left(\max\limits_{j\in[m]}\left\{\frac{p_j}{d_{ij}}\right\}\right) \\
    {\rm subject\ }{\rm to}& \quad\mathlarger{\sum}\limits_{j\in[m]}p_j=\mathlarger{\sum}\limits_{i\in[n]}B_i.
\end{array}
\end{equation*}
Still, the new objective function is difficult to tackle due to its nonsmooth nonconvex nature. To further simplify the problem, we remove the constraint by adding $-\sum_{i\in[n]}B_i\log(\sum_{j\in[m]}p_j)$ to the objective function, resulting in the following unconstrained problem: 
\begin{equation*}
    \min\limits_{p\in\R^m_+}\quad \mathlarger{\sum}\limits_{i\in[n]}B_i\log\left(\max_{j\in[m]}\left\{\frac{p_j}{d_{ij}}\right\}\right) -\mathlarger{\sum}\limits_{i\in[n]}B_i\log\left(\mathlarger{\sum}\limits_{j\in[m]}p_j\right).
\end{equation*}
This transformation is based on the observation that the new objective function, denoted by $\tilde{F}$, satisfies $\tilde{F}(tp)=\tilde{F}(p)$ for all $t>0$.
Finally, by replacing $p_j$ with $e^{\mu_j}$, we obtain the following problem:
\begin{equation}\label{eq:DC}
    \min\limits_{\mu\in\R^m}\quad \mathlarger{\sum}\limits_{i\in[n]}B_i\max_{j\in[m]}\mathlarger{\mathlarger{ \{ }}\mu_j-\log(d_{ij})\mathlarger{\mathlarger{ \} }} -\mathlarger{\sum}\limits_{i\in[n]}B_i\log\left(\mathlarger{\sum}\limits_{j\in[m]}e^{\mu_j}\right).\tag{DC}
\end{equation}
Let us define the objective function and components of \eqref{eq:DC} as below:
    \begin{align*}
 h_i(\mu)&\coloneqq B_i\max_{j\in[m]}\{\mu_j-\log(d_{ij})\},\    J_i(\mu)\coloneqq\{j:\mu_j-\log(d_{ij})=h_i(\mu)\}\\
 \ell(\mu)&\coloneqq \sum_{i\in[n]}B_i\log\left(\sum\limits_{j\in[m]}e^{\mu_j}\right),\ F(\mu)\coloneqq\sum_{i\in[n]}h_i(\mu)-\ell(\mu).
\end{align*}
Note that the piecewise linear functions $h_i,i\in[n]$ and the log-sum-exp function $\ell$ are convex by \cite[Example 3.5 and Page 74]{boyd2004convex}. The above problem is a \textit{difference-of-convex} (DC for short) program.  Interestingly, the first-order information of $F$ has its economic meanings. Given a vector $\mu\in\R^m$, we view $\nabla\ell(\mu)$ as the price vector $p$ and subgradient $v_i\in \partial h_i(\mu)$ as the income vector of agent $i$, i.e., the $j$-th component $v_{ij}=p_jx_{ij}$ is the income of agent $i$ from chore $j$. Then, the subgradient $u\coloneqq\sum_{i\in[n]}v_i-\nabla\ell(\mu)\in\partial F(\mu)$ corresponds to the value of unallocated chores; see \cite[Sec 4]{cole2017convex} for a similar discussion in the goods setting. 

We then characterize the relationship between stationary points of \eqref{eq:DC} and CE for chores. We first observe that the objective function is additively homogeneous. Given this property, we show that the stationary points of \eqref{eq:DC} have a one-to-one correspondence with CE for chores.
\begin{Fact}\label{fact:homo}
 $  F(\mu^*+t\1)= F(\mu^*)$  and $\partial F(\mu^*+t\1)=\partial F(\mu^*)$ for all $t\in\R$ and $\mu^*\in\R^m$.
\end{Fact}
\begin{theorem}\label{th:DC}
Every stationary point of problem \eqref{eq:DC} (denoted by $\mu^*$) and subgradients $v^*_i\in \partial h_i(\mu^*), i \in [n]$ that satisfy
\begin{equation}\label{eq:condition}
    \sum\limits_{j\in[m]}e^{\mu^*_j}=\sum_{i\in[n]}B_i\text{ and }\sum\limits_{i\in[n]}v^*_{ij}- {e^{\mu^*_j}} =0\text{ for all }j\in[m]
\end{equation}
correspond to a CE for chores with price $p^*_j=e^{\mu^*_j}$ and allocation $x^*_{ij}=v^*_{ij}/p^*_j$. 
\end{theorem}
Theorem \ref{th:DC} says that \eqref{eq:DC} serves as a formulation for computing CE for chores. This formulation, as we will show later, possesses interesting structural properties (such as K{\L} exponent of $1/2$) and can be tackled by classic optimization techniques (such as the smoothing method), thereby enabling us to design algorithms overcoming limitations (L1) and (L2).
\begin{proof}[Proof of Theorem \ref{th:DC}]
    \textit{From stationary points to CE:} Given a  pair of $(\mu^*,v^*)$ that satisfy \eqref{eq:condition}, we verify (E1)--(E3) of Definition \ref{def:CE}. First, we substitute $p^*_j=e^{\mu^*_j}$ and $x^*_{ij}=v^*_{ij}/p^*_j$ into the second equality of \eqref{eq:condition}. This gives $\sum_{i\in[n]}x^*_{ij}=1,$ which accords with (E3). 
   
   Second, by $v^*_i\in\partial h_i(\mu^*)=\conv\{B_ie_j:j\in J_i(\mu^*)\}$, we have $\sum_{j\in[m]}v^*_{ij}=B_i$. It follows that
   \[(p^*)^{\top}x^*_i=\sum_{j\in[m]}p^*_j\frac{v^*_{ij}}{p^*_j}=\sum_{j\in[m]}v^*_{ij}=B_i,\]
   which proves (E1). 
   
   Third, the relation $v^*_i\in\partial h_i(\mu^*)=\conv\{B_ie_j:j\in J_i(\mu^*)\}$ and the definition of $J_i(\mu^*)$ imply that whenever $v^*_{ij}>0$, 
   \[h_i(\mu^*)=B_i(\mu^*_j-\log(d_{ij}))\geq B_i(\mu^*_{j^{\prime}}-\log(d_{ij^{\prime}}))\quad\text{ for all }j^{\prime}\in[m].\] Using $p_j^*=e^{\mu_j^*}$, we further have
   \[\frac{p^*_j}{d_{ij}}\geq\frac{p^*_{j^{\prime}}}{d_{ij^{\prime}}}\text{ for all }j^{\prime}\in[m]\text{ whenever }v^*_{ij}>0. \]
   On the other hand, by $x^*_{ij}=v^*_{ij}/p^*_j$, we know that $x^*_{ij}>0$ if and only if $v^*_{ij}>0$. This, together with the above inequality, yields \eqref{eq:e2}. Then, by Fact \ref{le:e2}, the condition (E2) holds. We conclude that $(p^*,x^*)$ is a CE.
   
 \textit{From CE to stationary points:} Consider a CE with price $p^*$ and allocation $x^*$. By Corollary \ref{co:p}, we have $p^*>0$. Then, the relation $p^*_j=e^{\mu^*_j}$ and the abundant constraint $\sum_{j\in[m]}p^*_j=\sum_{i\in[n]}B_i$ ensure that there is a unique $\mu^*$ with $\mu^*_j=\log(p^*_j)$ satisfying $ \sum_{j\in[m]}e^{\mu^*_j}=\sum_{i\in[n]}B_i$. Further, the relation $x^*_{ij}=v^*_{ij}/p^*_j=v^*_{ij}/e^{\mu^*_j}$ and (E3) of Definition \ref{def:CE} imlpy the second equality of \eqref{eq:condition}. 
 
 It is left to show $v^*_i\in\partial h_i(\mu^*)$. Recall that $\partial h_i(\mu^*)=\conv\{B_ie_j:j\in J_i(\mu^*)\}$. It suffices to verify:
 \begin{gather}
     \label{eq:v1}
     \sum_{j\in[m]}v^*_{ij}=B_i \text{ and }v^*_{ij}\geq0\text{ for all }i\in[n],\\
     \label{eq:v2}
     v^*_{ij}>0\text{ implies }j\in J_i(\mu^*).
 \end{gather}
\eqref{eq:v1} directly follows from $x^*_{ij}=v^*_{ij}/p^*_j$ and (E1).
We focus on the proof of \eqref{eq:v2}. Due to the relation $x^*_{ij}=v^*_{ij}/p_j^*$, we know that $v^*_{ij}>0$ if and only if $x^*_{ij}>0$. Note that (E2) and Fact \ref{le:e2} ensure that $j\in\argmax_j\{{p_j^*}/{d_{ij}}\} $ whenever $x^*_{ij}>0$ and recall
 \begin{equation}\label{eq:Jmax}
J_i(\mu^*)=\argmax_j \{B_i(\mu^*_j-\log(d_{ij}))\}=\argmax_j\left\{\frac{p_j^*}{d_{ij}}\right\}.     
 \end{equation}
 We see that $j\in J_i(\mu^*)$ whenever $v^*_{ij}>0$, i.e., \eqref{eq:v2} holds. This completes the proof.
\end{proof}

\subsection{Error Bound and K{\L} Exponent of \eqref{eq:DC}}
In what follows, we show that \eqref{eq:DC} enjoys a local error bound and further the K{\L} property with exponent $1/2$. These findings are of independent interest. 
\begin{proposition}[Local Error Bound]\label{pro:EB} 
Let $\cU^*$ be the set of stationary points of \eqref{eq:DC} and $\bar{\mu}\in\cU^*$. There exist constants $\delta>0$ and $\tau>0$ such that for all $\mu$ with $\|\mu-\bar{\mu}\|\leq \delta$,
    \begin{equation*}
        \dist\left(\mu,\cU^*\right)\leq \tau\dist\left(0,\partial F(\mu)\right).
    \end{equation*}
\end{proposition}
\begin{theorem}[K{\L} exponent of $1/2$]\label{th:KL} Problem \eqref{eq:DC} satisfies the K{\L} property with exponent $1/2$, i.e., for every $\bar{\mu}\in\R^m$, there exist constants $\epsilon,\eta,\nu>0$ such that 
\[F(\mu)-F(\bar{\mu})\leq\eta\cdot \dist\left(0,\partial F(\mu)\right)^2\]
whenever $\|\mu-\bar{\mu}\|\leq\epsilon$ and $F(\bar{\mu})<F(\mu)<F(\bar{\mu})+\nu$.
\end{theorem}
The K{\L} property with exponent $1/2$ for \eqref{eq:DC} provides a solid basis for designing an algorithm with iterate convergence guarantee. Specifically, if an algorithm applied to \eqref{eq:DC} satisfies the so-called sufficient decrease property and relative error conditions (see, e.g., \cite[H1\&H2]{attouch2013convergence}), then its iterates would converge R-linearly\footnote{The sequence $\{\mu^k\}_{k\geq0}\subseteq\R^m$ is said to converge R-linearly to a vector $\bar{\mu}\in\R^m$ if there exist constants $\gamma>0$ and $\rho\in(0,1)$ such that $\|\mu^k-\bar{\mu}\|\leq \gamma\cdot \rho^k$; see, e.g., \cite[Sec. 1]{zhou2017unified}. } to a stationary point of \eqref{eq:DC}. In the next section, we shall present one such algorithm.
\section{Difference-of-convex Algorithm}\label{sec:dca}
We apply DCA to \eqref{eq:DC} to achieve R-linear convergence to a stationary point and, by Theorem \ref{th:DC}, a CE for chores. In addition to the iterate convergence guarantee, we aim to minimize the computational cost of solving subproblems. This motivates us to carefully design DCA based on the structure of \eqref{eq:DC}.

Instead of implementing DCA directly, we rewrite \eqref{eq:DC} by adding regularization terms to $\sum_{i\in[n]}h_i$ and $\ell$ respectively:
\begin{equation}\label{eq:DC-regular}
    \min\limits_{\mu\in\R^m}\quad \mathlarger{\sum}\limits_{i\in[n]}B_i\max_{j\in[m]}\mathlarger{\mathlarger{ \{ }}\mu_j-\log(d_{ij})\mathlarger{\mathlarger{ \} }} + \frac{\eta}{2}\|\mu\|^2 -\mathlarger{\sum}\limits_{i\in[n]}B_i\log\left(\mathlarger{\sum}\limits_{j\in[m]}e^{\mu_j}\right)-\frac{\eta}{2
    }\|\mu\|^2. 
\end{equation}
Here $\eta>0$ is the regularization coefficient. For simplicity, we define $$f(\mu) \coloneqq \mathlarger{\sum}\limits_{i\in[n]}B_i\max_{j\in[m]}\mathlarger{\mathlarger{ \{ }}\mu_j-\log(d_{ij})\mathlarger{\mathlarger{ \} }} + \frac{\eta}{2}\|\mu\|^2 \text{ and }g(\mu)\coloneqq\mathlarger{\sum}\limits_{i\in[n]}B_i\log\left(\mathlarger{\sum}\limits_{j\in[m]}e^{\mu_j}\right)+\frac{\eta}{2
    }\|\mu\|^2.$$
In the $k$-th iteration of DCA,  we solve the following strongly convex problem to obtain $\mu^{k+1}$:
\begin{equation}\label{eq:SQP}
    \min\limits_{\mu\in\R^m}\quad f(\mu) -\nabla g(\mu^k)^{\top}(\mu - \mu^k).
\end{equation}
Despite the strong convexity of \eqref{eq:SQP}, the nonsmooth nature of $f$ precludes the applicability of fast algorithms. This motivates us to replace the nonsmooth terms $\max_{j\in[m]}\{\mu_j-\log(d_{ij})\}$ with auxiliary variables $t_i$ and constraints $\mu_j-\log(d_{ij})\leq t_i$, leading to the following {quadratic program} (QP):
\begin{equation}\label{eq:SQP-R}
\begin{aligned}
    \min \limits_{\mu\in\R^m} & \quad \mathlarger{\sum}\limits_{i\in[n]}B_i t_i + \frac{\eta}{2}\|\mu\|^2 -\nabla g(\mu^k)^\top(\mu - \mu^k)\\
    {\rm subject\ } & {\rm to}  \quad \mu_j - t_i \leq \log (d_{ij}),\qquad \forall\ i \in [n], j \in [m].
\end{aligned}   
\end{equation}
We further derive the dual of \eqref{eq:SQP-R} to obtain a well-structured subproblem. Using the KKT conditions, we can recover $\mu^{k+1}$ from the dual optimal solution $\lambda^k$. 
\begin{proposition}[Subproblem of DCA]\label{pro:subproblem} Solving problem \eqref{eq:SQP} is equivalent to solving
\begin{equation}\label{eq:SQP-Dual}
\begin{aligned}
    \min \limits_{\lambda\in\R^{m\times n}_+}  \quad & \frac{1}{2}\left \| \mathlarger{\sum}\limits_{i\in[n]}  \lambda_{i}-\frac1{\eta}\nabla g(\mu^k)\right\|^2 + \log(D)\bullet \lambda\\
    {\rm subject\ }  {\rm to} \quad & \mathlarger{\sum}\limits_{j\in[m]} \lambda_{ij} = \frac1{\eta}B_i,\qquad \forall\ i \in [n],
\end{aligned}
\end{equation}
where $\bullet $ refers to the Frobenius inner product of two matrices. The optimal solution of \eqref{eq:SQP} (denoted by $\mu^{k+1}$) can be recovered from that of \eqref{eq:SQP-Dual} (denoted by $\lambda^k$) via
$\mu^{k+1}=\frac{1}{\eta}\nabla g(\mu^k)-(\lambda^k)^{\top}\1.$
\end{proposition}
\begin{remark}
    Problem \eqref{eq:SQP-Dual} is essentially a convex QP with simplex constraints. It can be efficiently solved by first-order methods like mirror descent and PGD, where the gradient of the objective can be evaluated in $\cO(mn)$ time and projection onto the feasible set only requires $\cO(mn)$ \cite{duchi2008efficient} (resp. $\cO(mn\log(n))$ \cite{beck2003mirror}) time for mirror descent (resp. PGD). 
\end{remark}
\begin{proof}[Proof of Proposition \ref{pro:subproblem}]
Let the variables $\lambda_{ij}\geq0,i\in[n],j\in[m]$ be the Lagrange multipliers associated with the constraints $ \mu_j - t_i \leq \log(d_{ij})$ of the primal problem \eqref{eq:SQP-R}. The associated Lagrangian $ \mathcal{L} $ can be expressed as
\[
\mathcal{L}(\mu, t, \lambda) \coloneqq \sum_{i \in [n]} B_i t_i + \frac{\eta}{2} \|\mu\|^2 - \nabla g(\mu^k)^\top (\mu - \mu^k) + \sum_{i \in [n]} \sum_{j \in [m]} \lambda_{ij} (\mu_j - t_i - \log(d_{ij})).
\]
To find the infimum $g(\lambda) = \inf_{\mu,t} \mathcal{L}(\mu, t, \lambda)$, we compute the gradients of $ \mathcal{L} $ with respect to the primal variables $ \mu_{j},j\in [m] $ and $t_{i},i \in [n]$, and set them as zero:
    \begin{align}
      \label{eq:KKT-subDCA1}
        \frac{\partial \mathcal{L}}{\partial \mu_j} &= \eta \mu_j - \nabla_j g(\mu^k) + \sum_{i \in [n]} \lambda_{ij} = 0, \quad \forall\ j \in [m],\\
          \label{eq:KKT-subDCA2}
        \frac{\partial \mathcal{L}}{\partial t_i} &= B_i - \sum_{j \in [m]} \lambda_{ij} = 0, \quad \forall\ i \in [n].
    \end{align}
We then derive the dual problem, which is of the form $\max_{\lambda\in\R^{n\times m}_+} g(\lambda)$. By substituting \eqref{eq:KKT-subDCA1} into the expression of $\mathcal{L}$ and adding the constraints \eqref{eq:KKT-subDCA2}, we obtain the following dual problem:
\begin{equation*}
 \begin{aligned}
    \max_{\lambda\in\R^{m\times n}_+} & \quad -\frac{1}{2\eta} \sum_{j \in [m]} \left( \nabla_j g(\mu^k) - \sum_{i \in [n]} \lambda_{ij} \right)^2 - \sum_{i \in [n]} \sum_{j \in [m]} \lambda_{ij} \log(d_{ij}), \\
    \text{subject to} & \quad \sum_{j \in [m]} \lambda_{ij} = B_i, \quad \forall\ i \in [n].
\end{aligned}   
\end{equation*}
Replace $\lambda_{ij}$ with $\frac1{\eta} \lambda_{ij}$ and derive the equivalent minimization problem. We obtain the desired problem \eqref{eq:SQP-Dual}.
\end{proof}
\subsection{Linear Rate of DCA}
Now, we turn to show that the iterates of the proposed DCA converge R-linearly to a stationary point of \eqref{eq:DC}. As a byproduct, we show that their approximate CE measures (which will be defined below) also converge R-linearly to zero. 
To begin, let us develop the relation between approximate CE and approximate stationary points of \eqref{eq:DC} (which refer to, roughly speaking, the points $\mu\in\R^m$ with small subgradient norm $\dist(0,\partial F(\mu))$). This relation involves the following functions that will be repeatedly used in our analysis:
\[ q_j(\mu)\coloneqq \sum_{i\in[n]}B_i\cdot\frac{{e^{\mu_j}}}{\sum_{j\in[m]}e^{\mu_j}},\qquad j\in[m].\]
\begin{lemma}\label{le:dist}
  Suppose that there exists a vector $u=\sum_{i\in[n]}v_i-\nabla\ell(\mu)\in\partial F(\mu)$, where $v_i\in\partial h_i(\mu)$, satisfying $|u_j/q_j(\mu)|\leq \epsilon$ for all $j\in[m]$. Then, the price-allocation pair $(p,x)$ with $p_j=q_j(\mu)$ and $x_{ij}=v_{ij}/p_j$ is an $\epsilon$-CE.
\end{lemma}
Then, let us state our main result. 
\begin{theorem}\label{th:linear_DCA}
    Let $\{\mu^k\}_{k\geq0}$ be the sequence generated by the DCA with regularization coefficient $\eta$. Then, the sequence $\{\mu^k\}_{k\geq0}$ converges R-linearly to a stationary point of \eqref{eq:DC}. Also, the measure sequence $\{|u^k_j/q_j(\mu^k)|\}_{k\geq1}$, where $u^k\coloneqq\nabla g(\mu^{k-1})-\nabla g(\mu^{k})\in\partial F(\mu^{k})$, converges R-linearly to zero.
\end{theorem}
Since stationary points of \eqref{eq:DC} are CE by Theorem \ref{th:DC}, Theorem \ref{th:linear_DCA} implies that the DCA sequence converges to an exact CE, thereby resolving (L1). We highlight that this iterate convergence result, which relies on Theorem \ref{th:KL}, is significantly more difficult to obtain than those in \cite{gao2020first,nan2024convergence} due to the nonconvexity in the chores setting. Furthermore, by adding a rounding procedure, the DCA has a non-asymptotic rate of $\tilde{\cO}(n/\epsilon^2)$ for finding an $\epsilon$-CE without compromising R-linear convergence. See Appendix \ref{appen:round_dca} for details.
\subsection{Proof of Theorem \ref{th:linear_DCA}}
We first develop the sufficient descent and relative error properties of DCA.
\begin{lemma}[Sufficient Descent \& Relative Error of DCA]\label{le:suff} Let $\{\mu^k\}_{k\geq0}$ be the iterates of DCA, i.e., $\mu^{k+1}$ is the optimal solution of \eqref{eq:SQP} for all $k\geq0$. It holds that
\begin{equation}\label{eq:suff}
    F(\mu^{k+1})-F(\mu^k)\leq -\frac{\eta}{2}\|\mu^k-\mu^{k+1}\|^2 \tag{Sufficient Descent};
\end{equation}
\begin{equation}\label{eq:relative}
    \exists\ u^{k+1}\in\partial F(\mu^{k+1})\text{ such that }\|u^{k+1}\|\leq \left(\eta+\sum_{i\in[n]}B_i\right)\|\mu^{k+1}-\mu^k\| \tag{Relative Error}.
\end{equation}
In particular, one can choose $u^{k+1}=\nabla g(\mu^k)-\nabla g(\mu^{k+1})$ here. In this case, if $\|\mu^{k+1}-\mu^k\|\leq \gamma$ for some $\gamma>0$, then there is a tighter relative error:
\begin{equation}\label{eq:tighter}
    |u_j^{k+1}|\leq \left(\eta+e^{2\gamma}q_j(\mu^{k+1})\right)\|\mu^{k+1}-\mu^k\|.
\end{equation}
\end{lemma}
\begin{proof}
    \textit{For sufficient descent:}
    By the optimality of $\mu^{k+1}$ for \eqref{eq:SQP}, we have
    \[f(\mu^{k+1})-\nabla g(\mu^k)^{\top}(\mu^{k+1}-\mu^k)\leq f(\mu^k).\]
    Noticing that the function $g$ is $\eta$-strongly convex, we see that  
    \[\nabla g(\mu^k)^{\top}(\mu^{k+1}-\mu^k)\leq g(\mu^{k+1})-g(\mu^k)-\frac{\eta}{2}\|\mu^k-\mu^{k+1}\|^2.\]
    Combining the above two inequalities gives 
    $f(\mu^{k+1})\leq f(\mu^k)+g(\mu^{k+1})-g(\mu^k)-\frac{\eta}{2}\|\mu^k-\mu^{k+1}\|^2$.
    This directly implies \eqref{eq:suff} as $F=f-g$.

\textit{For relative error:} The optimality condition of \eqref{eq:SQP} gives 
$0\in\partial f(\mu^{k+1})-\nabla g(\mu^k).$
It follows that \[\nabla g(\mu^k)-\nabla g(\mu^{k+1})\in\partial f(\mu^{k+1})-\nabla g(\mu^{k+1}).\] Let $u^{k+1}\coloneqq\nabla g(\mu^k)-\nabla g(\mu^{k+1})$. We have $u^{k+1}\in\partial F(\mu^{k+1})$. Moreover, by direct computation, we see that
\begin{equation}\label{eq:rlu}
    u^{k+1} = \eta(\mu^k-\mu^{k+1})+\nabla\ell(\mu^k)-\nabla\ell(\mu^{k+1}).
\end{equation}
Note that \cite[Page 71]{boyd2004convex} gives $0\preceq\nabla^2\ell(\mu)\preceq\sum_{i\in[n]}B_i\cdot\diag(e^{\mu_1},\ldots,e^{\mu_m})/\sum_{i\in[m]}e^{\mu_j}\preceq(\sum_{i\in[n]}B_i)I_m$. We have $\|\nabla\ell(\mu^k)-\nabla\ell(\mu^{k+1})\|\leq\sum_{i\in[n]}B_i\|\mu^k-\mu^{k+1}\|$. This, together with \eqref{eq:rlu}, yields  \eqref{eq:relative}.


It is left to show \eqref{eq:tighter} under the condition $\|\mu^{k+1}-\mu^k\|\leq \gamma$. Let $\mu^{k,s}\coloneqq s\mu^k+(1-s)\mu^{k+1}$ for $s\in[0,1]$. Applying the mean value theorem to $\nabla_j\ell(\mu^{k})-\nabla_j\ell(\mu^{k+1})$, where $\nabla_j\ell$ is the $j$-th component of $\nabla\ell$, we have 
\begin{equation}\label{eq:mean}
    \nabla_j\ell(\mu^{k})-\nabla_j\ell(\mu^{k+1})=\nabla\nabla_j\ell(\mu^{k,s})^{\top}(\mu^k-\mu^{k+1})\quad\text{ for some }s\in(0,1).
\end{equation}
By direct computation, we have $\|\nabla\nabla_j\ell(\mu)\|\leq 2 q_j(\mu)$. This, together with  $\|\mu^{k+1}-\mu^k\|\leq \gamma$, implies 
\begin{equation}\label{eq:mean2}
    \|\nabla\nabla_j\ell(\mu^{k,s})\|\leq 2 q_j(\mu^s)=2\sum_{i\in[n]}B_i\frac{e^{\mu^s_j}}{\sum_{j\in[m]}e^{\mu^s_j}}\leq 2\sum_{i\in[n]}B_i\frac{e^{\gamma+\mu^{k+1}_j}}{\sum_{j\in[m]}e^{\mu^{k+1}_j-\gamma}}=e^{2\gamma}q_j(\mu_{k+1}).
\end{equation}
Now, combining \eqref{eq:mean} and \eqref{eq:mean2}, and using the Cauchy inequality, we have
\begin{equation*}
   | \nabla_j\ell(\mu^{k})-\nabla_j\ell(\mu^{k+1})|\leq\|\nabla\nabla_j\ell(\mu^{k,s})\|\cdot\|\mu^{k+1}-\mu^k\|\leq e^{2\gamma}q_j(\mu_{k+1})\|\mu^{k+1}-\mu^k\|.
\end{equation*}
This, together with \eqref{eq:rlu} and the triangle inequality, yields \eqref{eq:tighter}, which completes the proof.
\end{proof}

Second, we establish the boundedness of DCA iterates; see Appendix \ref{appen:dca} for the proof.
\begin{lemma}[Lower Bound for Iterates of DCA]\label{le:bound}
The sequence $\{\mu^k\}_{k\geq0}$ generated by the DCA is bounded with $\1^{\top}\mu^k\equiv \1^{\top}\mu^0$.
\end{lemma}
Now, equipped with the fact that \eqref{eq:DC} has a K{\L} exponent of $1/2$, sufficient descent and relative error properties of DCA, and boundedness of $\{\mu^k\}_{k\geq0}$, using \cite[Theorem 2.3]{schneider2015convergence}, we derive the linear convergence of DCA iterates $\{\mu^k\}$ to some stationary point $\mu^*$, i.e., $\mu^k\to\mu^*$ with $\|\mu^k-\mu^*\|\leq c\rho^k$ for some $c>0$, $\rho\in(0,1)$. Furthermore, this iterate linear convergence result, together with the relative error, implies that for $k\geq1$,
\begin{equation*}\label{eq:linear1}
\begin{aligned}
     \|u^{k}\|\leq \left(\eta+\sum_{i\in[n]}B_i\right) (\|\mu^{k-1}-\mu^*\|+\|\mu^{k}-\mu^*\|  )\leq 2\left(\eta+\sum_{i\in[n]}B_i\right)c \rho^{k-1}.
\end{aligned}
\end{equation*}
Note that the boundedness of $\{\mu^k\}_{k\geq0}$ and continuity of $q_j$ imply that $q_j(\mu^k)\geq\delta$ for some $\delta>0$  for all $k\geq0$ and $j\in[m]$. The R-linear convergence of $\{|u^k_j/q_j(\mu^k)|\}_{k\geq1}$ directly follows. We complete the proof of Theorem \ref{th:linear_DCA}.
 
\section{Smoothed Gradient Method with Rounding Procedure}\label{sec:sgr}
Although the DCA addresses limitation (L1), its subproblem remains an optimization problem similar to those of EPM and GFW methods. To address (L2), we propose a subproblem-free method to compute approximate CE. The key is to find a suitable smooth approximation of the nonsmooth part of \eqref{eq:DC}. Specifically, we smooth the $\max$ terms in \eqref{eq:DC} via entropy regularization with parameter $\delta>0$:
\begin{equation}\label{eq:max_smooth}
    \min\limits_{\mu\in\R^m}\quad \mathlarger{\sum}\limits_{i\in[n]}B_i \max_{\mathbf{\lambda}_{i}\in\Delta_m }\left\{\mathlarger{\sum}_{j\in[m]}\lambda_{ij}\left(\mu_j-\log(d_{ij}) \right) -\delta \lambda_{ij}\log(\lambda_{ij}) \right\} -\mathlarger{\sum}\limits_{i\in[n]}B_i\log\left(\mathlarger{\sum}\limits_{j\in[m]}e^{\mu_j}\right).
\end{equation}
Here, the inner maximization problem has the closed-form solution  \[\lambda_{ij}^*=\frac{e^{\frac{\mu_j-\log(d_{ij})}{\delta}} }{\sum_{j\in[m]}e^{\frac{\mu_j-\log(d_{ij})}{\delta}} },\quad\ i\in[n],j\in[m].\]
Substituting the optimal $\lambda^*$ into \eqref{eq:max_smooth}, we obtain the equivalent problem:
\begin{equation}\label{eq:SP}
    \min\limits_{\mu\in\R^m}\quad F_{\delta}(\mu)\coloneqq \delta\mathlarger{\sum}\limits_{i\in[n]}B_i  \log\left(\mathlarger{\sum}_{j\in[m]} e^{\frac{\mu_j-\log(d_{ij})}{\delta}} \right)  -\mathlarger{\sum}\limits_{i\in[n]}B_i\log\left(\mathlarger{\sum}\limits_{j\in[m]}e^{\mu_j}\right). \tag{SP}
\end{equation} 
Clearly,  $F\leq F_{\delta}\leq F+\delta\log(m) \sum_{i\in[n]}B_i$ by the $\max$ operator in \eqref{eq:max_smooth} and $0\leq -\sum_{j\in[m]}\delta\lambda_{ij}\log(\lambda_{ij})\leq \delta\log(m)$ for $\lambda_i\in\Delta_m$. It follows that \eqref{eq:SP} is a smooth approximation of \eqref{eq:DC}. Now, two natural questions arise: (i) How do we derive an $\epsilon$-CE from \eqref{eq:SP}? (ii) Does the function $F_\delta$ possess any favorable properties?
For question (ii), consider the gradient $\nabla F_{\delta}(\mu)$, whose $j$-th component is
\[\nabla_jF_{\delta}(\mu)=\mathlarger{\sum}\limits_{i\in[n]}B_i\left(\frac{e^{\frac{\mu_j-\log(d_{ij})}{\delta}}}{\sum\limits_{j\in[m]} e^{\frac{\mu_j-\log(d_{ij})}{\delta}} } \right)-\mathlarger{\sum}\limits_{i\in[n]}B_i\left(\frac{e^{\mu_j}}{{\sum}_{j\in[m]}e^{\mu_j}}\right). \]
By direct computation, we see that $\nabla F_{\delta}$ is $\sum_{i\in[n]}B_i(1/\delta+1)$-Lipschitz continuous and $\nabla F_{\delta}(\mu)$ can be evaluated in $\cO(mn)$ time. We summarize the properties of $F_{\delta}$ as follows:
\begin{Fact}[Properties of $F_{\delta}$]\label{fact:F_epsilon}  The following properties hold:
\begin{enumerate}[{\rm (i)}]
    \item $F\leq F_{\delta}\leq F+\delta\log(m) \sum_{i\in[n]}B_i$;
    \item $\nabla F_{\delta}(\mu)$ can be evaluated in $\cO(mn)$ time;
    \item $\nabla F_{\delta}$ is $\sum_{i\in[n]}B_i(1/\delta+1)$-Lipschitz continuous.
\end{enumerate}
\end{Fact}
\begin{algorithm}[htbp]
	\caption{Round$(a,\mu^0,b)$ }
	\begin{algorithmic}[1]
		\REQUIRE{ Parameter $a\in(-\infty,\underline{\mu_{\delta}}]$, initial point $\mu^0\in\R^m$, $b=\sum_{i\in[n]}B_i$  }
        \STATE $k=1$, $\mu^{out}=\mu^0$, $J_0=\{j:q_j(\mu^0)<e^{a}\}$, $\mu^1_j=\max_{j\in J_{0}}\{\mu^0_j\}$ for $j\in J_0$
        \WHILE{ $J_{k-1}\neq\emptyset$ }
        \STATE  ${\rm thres }_k= \log(\sum_{j\in[m]\setminus J_{k-1}}e^{\mu^k_j})-\log(b\cdot e^{-a }-|J_{k-1}| )$
		\STATE $\Delta J_k = \argmin_{j}\{\mu^k_j:\mu^k_j<{\rm thres}_k,j\notin J_{k-1} \}$
        \IF{ $\Delta J_k\neq\emptyset$ }
        \STATE $c_k=\mu^k_j$ for $j\in\Delta J_k$
         \STATE $\mu^{k+1}_j=\max\{\mu^k_j,c_k\}$ \text{ for }$j\in [m]$
         \STATE  $J_k=J_{k-1}\cup \Delta J_k$, $k=k+1$
         \ELSE
         \STATE $\mu^{new}_j=\max\{\mu^k_j, {\rm thres}_k \}$ for $j\in [m]$
         \STATE $d=\1^{\top}(\mu^0-\mu^{new})/m$ and $\mu^{out}=\mu^{new}+d\1$ 
         \STATE \textbf{break}
        \ENDIF
		\ENDWHILE{ }
		\ENSURE{$\mu^{out}$}
	\end{algorithmic}
	\label{al:round} 
\end{algorithm}
For question (i), we develop the relation between approximate CE and approximate stationary points of \eqref{eq:SP} (which refer to, roughly speaking, the points $\mu\in\R^m$ with small gadient norm $\|\nabla F_\delta(\mu)\|$). We recall the functions
$q_j(\mu)\coloneqq \sum_{i\in[n]}B_i{{e^{\mu_j}}}/{\sum_{j\in[m]}e^{\mu_j}}$, $j\in[m]$ and present the following lemma.
\begin{lemma}\label{le:SP_CE}
       Let the parameter $\epsilon>0$ be fixed. Suppose that the vector $\mu\in\R^m$ and the parameter $\delta>0$ satisfy 
  \[\left|\frac{\nabla_j F_{\delta}(\mu)}{q_j(\mu)}\right|\leq \epsilon\ \text{for all}\ j\in[m];\quad  \epsilon\geq (1.3+\log(m-1))\delta.\]
      Then, the pair $(p,v)$ with $p_j=q_j(\mu)$ and $x_{ij}=v_{ij}/p_j$ is an $\epsilon$-CE, where $v_{ij}={B_ie^{\frac{\mu_j-\log(d_{ij})}{\delta}} }/{{\sum}_{j\in[m]} e^{\frac{\mu_j-\log(d_{ij})}{\delta } } }.$
\end{lemma}
By Lemma \ref{le:SP_CE}, we need to seek a vector $\mu$ satisfying $\max_{j\in[m]}|\nabla_j F_{\delta}(\mu)/q_j(\mu)|\leq\epsilon$ to obtain an $\epsilon$-CE. A candidate algorithm should not only search for $\mu$ with small gradient norm $\|\nabla F_{\delta}(\mu)\|$ but also provide lower bounds on $\{q_j(\mu^k)\}_{k\geq0}$ for its iterates $\{\mu^k\}_{k\geq0}$. To obtain such lower bounds, we design a rounding procedure that rounds $\mu$ into $\hat{\mu}$ satisfying $q_{j}(\hat{\mu})\geq e^a,j\in[m]$ for some $a\in\R$ without increasing the function value. We first present an observation that serves as the basis of the rounding.
\begin{Fact}\label{fact:underlined2}
   Suppose that $\delta\leq 1/(2+\log(m-1))$ and there is an index $j_0\in[m]$ such that $q_{j_0}(\mu)<e^{\underline{\mu_{\delta}} }$, where
    \[\underline{\mu_{\delta}}\coloneqq\log\left(\frac{\sum_{i\in[n]}B_i }{2m}\right)-\frac{1+\delta}{1-\delta}\log\left(\frac{\max_{i,j}\{d_{ij}\}}{\min_{i,j}\{d_{ij}\}}\right) -\delta \log(4m).\]
    Then, the $j_0$-th coordinate of approximate function gradient is negative, i.e.,  $\nabla_{j_0}F_{\delta}(\mu)<0$.
\end{Fact}
By Fact \ref{fact:underlined2} and the mean value theorem, if $q_j(\mu)<e^{\underline{\mu_{\delta}}}$ for some $j\in[m]$, then  $\mu_j$ can be increased to $\underline{\mu_{\delta}}$ without increasing the value of $F_{\delta}$. Further, intuitively, for every $\mu\in\R^m$, we can find an associated $\hat{\mu}$ such that $F_{\delta}(\hat{\mu})\leq F_{\delta}(\mu)$ and $q_j(\hat{\mu})\geq e^{\underline{\mu_{\delta}}}$ by increasing the coordinates $\mu_j$ to a \textit{threshold} for those indices $j$ that satisfy $q_j(\mu)<e^{\underline{\mu_{\delta}}}$.
This intuition leads us to the rounding procedure in Algorithm \ref{al:round}. Algorithm \ref{al:round} first identifies the indices $j$ that satisfy $q_j(\mu^0)<e^a$ to form an index set $J_0$, where $\mu^0\in\R^m$ is the input and $a\in\R$ is a constant smaller than ${\underline{\mu_{\delta}}}$. Then, Algorithm \ref{al:round} aims to increase $\mu_j,j\in J_0$ to a threshold number ${\rm thres}_0\in\R$ so that $q_j(\mu)\geq e^a$ for $j\in J_0$. However, such an aggressive increase in $\mu_j,j\in J_0$  may cause $q_j(\mu)< e^a$ for some $j\in[m]\setminus J_0$. To avoid this, the coordinates $\mu_j,j\in J_0$ are increased without exceeding the minimal value of $\mu_j,j\in [m]\setminus J_0$. By repeating this process, within $m$ steps Algorithm \ref{al:round} will stop and provide a desired output.
\begin{lemma}\label{le:round}
Consider $\delta$ and $\underline{\mu_{\delta}}$ in Fact \ref{fact:underlined2}.
Then,  Algorithm \ref{al:round} stops in $m$ steps and outputs a vector $\mu\in\R^m$ that satisfies
\begin{enumerate}[\rm (i)]   
    \item $\1^{\top}\mu=\1^{\top}\mu^0$;
    \item  $F_{\delta}(\mu)\leq F_{\delta}(\mu^0)$; 
    \item $q_j(\mu)\geq e^{a}$ for all $j\in[m]$.
\end{enumerate}
\end{lemma}
Now, we are ready to give an algorithm, i.e., the SGR, to compute an $\epsilon$-CE from \eqref{eq:SP}. In each iteration, we perform a gradient descent step for \eqref{eq:SP} to decrease the value of $F_{\delta}$ and round the iterate to obtain lower bounds on $\{q_j(\hat{\mu}^k)\}_{k\geq0},j\in[m]$. When the stopping criterion $|\nabla_j F_{\delta}(\hat{\mu}^k)/q_j(\hat{\mu}^k)|<\epsilon$, $j\in[m]$  is met, we can recover an $\epsilon$-CE from the output $\hat{\mu}^k$ due to Lemma \ref{le:SP_CE}.
\begin{algorithm}[htbp]
	\caption{Smoothed Gradient Method with Rounding (SGR) }
	\begin{algorithmic}[1]
		\REQUIRE{ Parameters $\epsilon$, $\delta$, and $\underline{\mu_{\delta}}$ satisfying conditions of Lemma \ref{le:SP_CE} and Fact \ref{fact:underlined2}, initial point $\mu^0\in\R^m$, maximum iteration number $k_{max}$, stepsize $\eta$  }
        \STATE $k=0$, $a=\underline{\mu_{\delta}}$, $b=\sum_{i\in[n]}B_i$
        \WHILE{ $k\leq k_{max}$ }
        \STATE $\hat{\mu}^k=$Round$(a,\mu^k,b)$
        \IF{ $|\nabla_j F_{\delta}(\hat{\mu}^k)/q_j(\hat{\mu}^k)|<\epsilon$ for all $j\in[m]$ }
          \STATE  \textbf{break}
         \ENDIF
        \STATE $\mu^{k+1}=\hat{\mu}^k-\eta\nabla F_{\delta}(\hat{\mu}^k)$
        \STATE $k=k+1$
        \ENDWHILE{ }
		\ENSURE{$\hat{\mu}^k$}
	\end{algorithmic}
	\label{al:sgr} 
\end{algorithm}

By exploring the sufficient decrease property for the SGR, we can further give an upper bound on the iteration complexity of SGR.  
\begin{theorem}\label{th:SGM}
Suppose that the parameters $\epsilon,\delta,$ and $\eta$ of SGR satisfy $0<\epsilon<1/2$, 
$\delta=\frac{\epsilon}{1.3+\log(m-1)}, \text{ and } \eta= \frac{\gamma}{\sum_{i\in[n]}B_i}\frac{\delta}{1+\delta}$ with $0<\gamma<1$.
 Then, the SGR finds an $\epsilon$-CE in at most $\tilde{\cO}(\frac{m^2}{\epsilon^3})$ iterations, and the total complexity is at most $\tilde{\cO}(\frac{m^3(n+m)}{\epsilon^3})$.  
\end{theorem}
Though the iteration complexity of SGR is higher than that of the GFW method, the SGR only computes the gradient of $F_{\delta}$ (at a time cost of $\cO(nm)$) and rounds the iterate (at a time cost of $\cO(m^2)$) in each iteration. By contrast, the GFW method needs to solve an LP with $mn$ constraints, which requires $\cO((mn)^{3.5})$ time \cite{vavasis1996primal}. Therefore, SGR offers a reasonable trade-off. We emphasize that the SGR is the first algorithm that computes approximate CE for chores without solving other optimization problems as subproblems, making it especially appealing in high-dimensional settings.
\section{Numerical Experiments}\label{sec:numerical}
We compare the numerical performance of the proposed DCA and SGR with the state-of-the-art GFW method proposed in \cite{chaudhury2024competitive}. Specifically, we evaluate the time they spend to compute an $\epsilon$-CE for chores, especially in \textit{high-dimensional} settings. We use different distributions to generate i.i.d. disutilities $d_{ij}$ and expected amounts $B_i$, where $i\in[n],j\in[m]$. These distributions include the uniform distribution on $(0,1]$, log-standard-normal distribution, exponential distribution with the scale parameter $1$, and uniform random integers on $\{1,2, \ldots, 1000\}$. We run all experiments on a personal desktop that uses the Apple M3 Pro Chip and has 18GB RAM. We use Gurobi version 11.0.0 to solve GFW subproblems. See Appendix \ref{appen:additional_exp} for more details.

First, we set $m=n$ and demonstrate the algorithms' CPU time in Figure \ref{fig:low_precision}. Figure \ref{fig:low_precision} shows that the CPU time of DCA and SGR is significantly lower than that of GFW: For all tested $n$, the CPU time of DCA (resp. SGR) does not exceed $50$ (resp. $10$) seconds, while the GFW may need $200+$ seconds. Moreover, the CPU time of SGR increases very slightly compared with the GFW as $n$ grows, which is consistent with our theory. Then, we consider the imbalanced-dimension case, where the ratio $n/m$ is large. In this setting, the superiority of our algorithms is even clearer: The DCA (resp. SGR) can be $7$ (resp. $50$) times faster than the GFW method. These results validate the practicality and efficiency of our algorithms. 
\begin{figure}[htbp]
    \centering
    \captionsetup[subfloat]{labelformat=empty} 
    \subfloat[Uniform]{\includegraphics[scale=0.22]{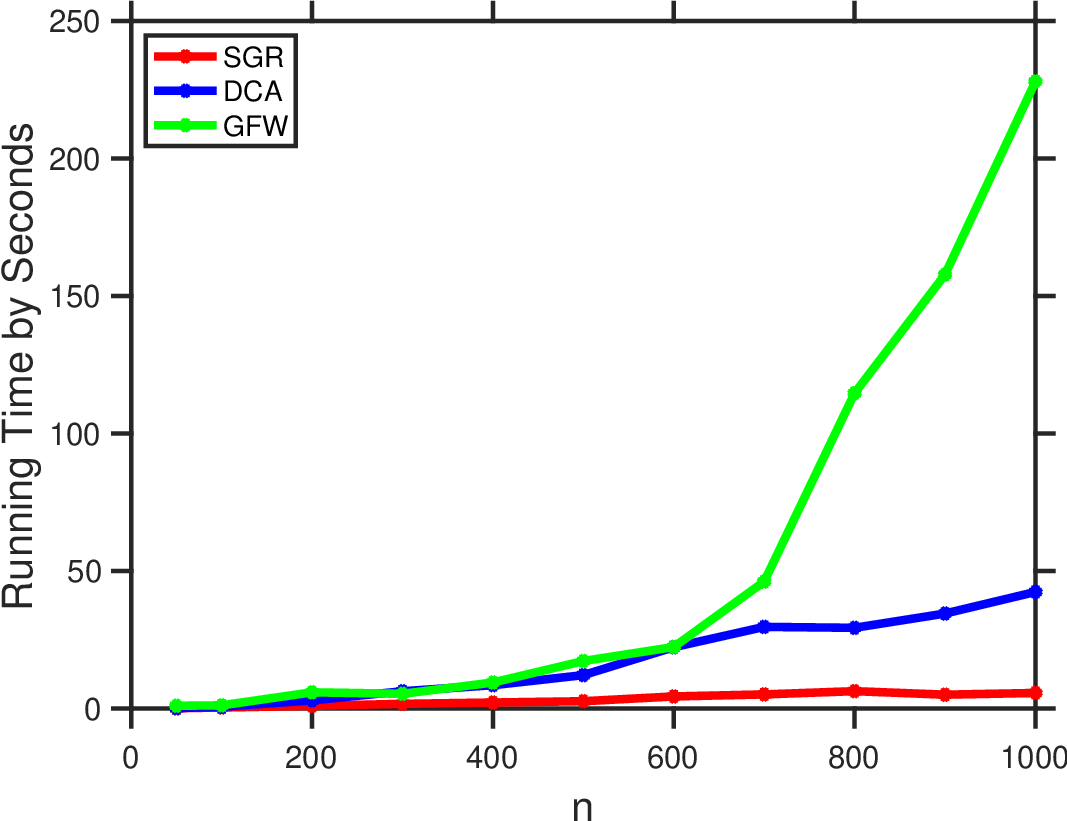}}\hspace{1mm}
    \subfloat[Log-Normal]{\includegraphics[scale=0.22]{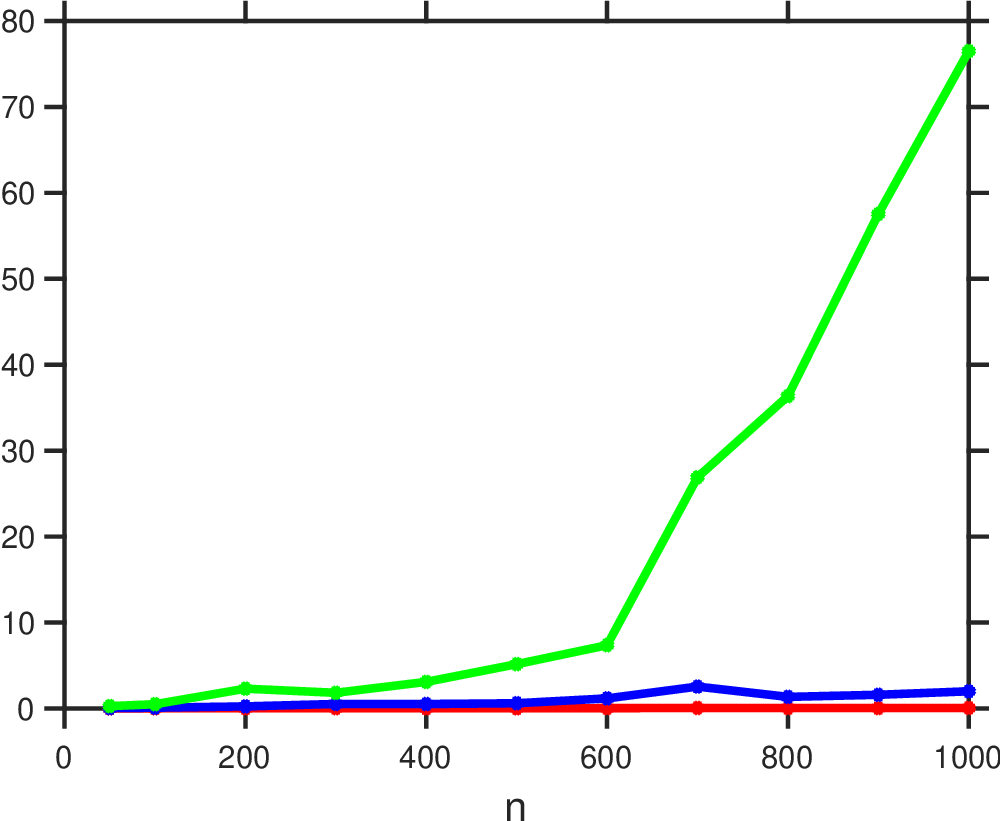}}\hspace{1mm}
    \subfloat[Exponential]{\includegraphics[scale=0.22]{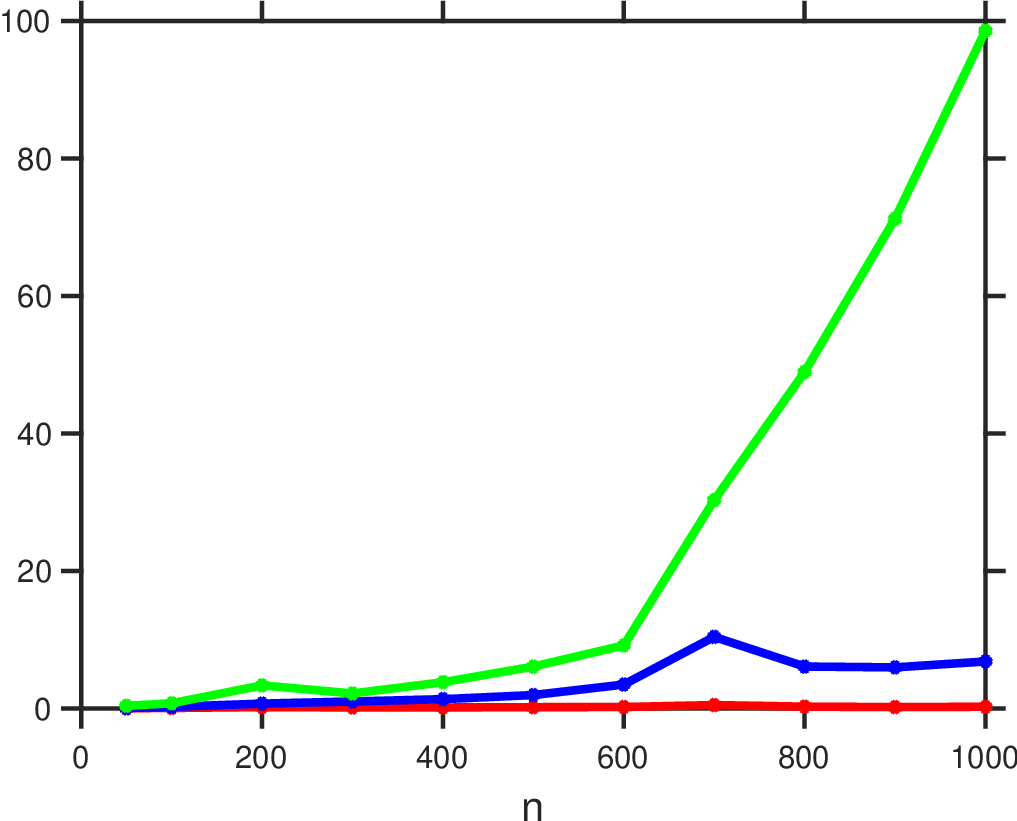}}\hspace{1mm}
    \subfloat[Integer]{\includegraphics[scale=0.22]{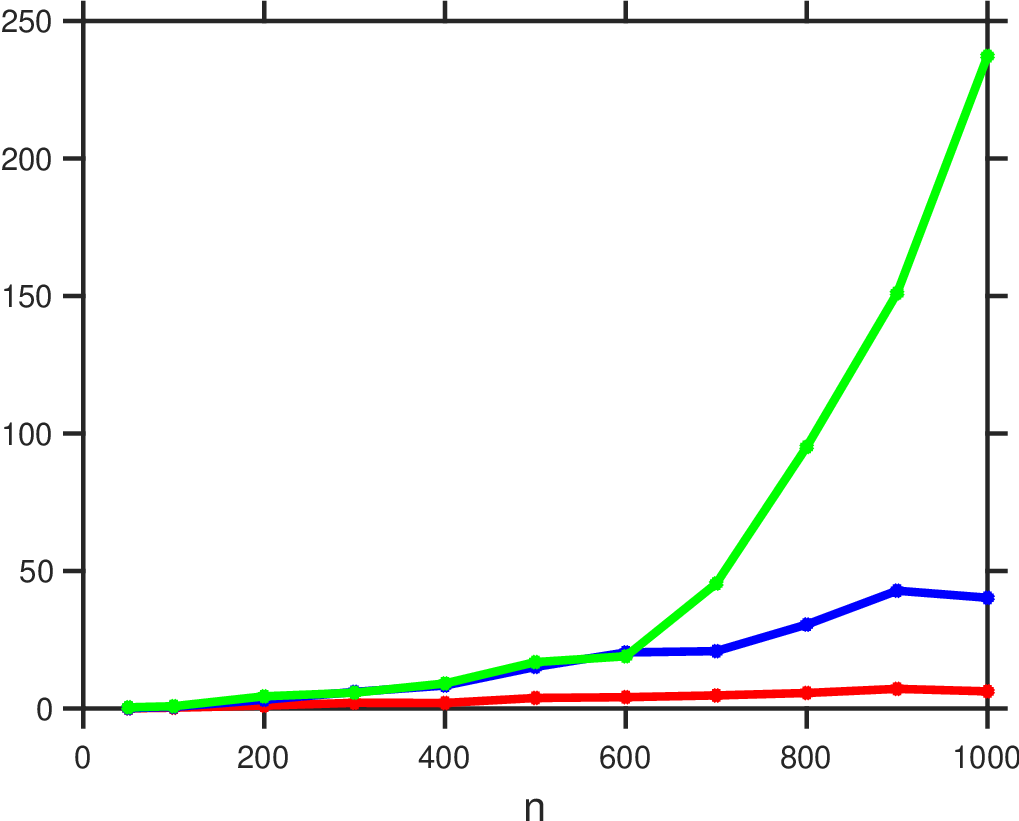}}
    \caption{CPU Time Comparison under Different Generative Models, $\epsilon=0.01$}
    \label{fig:low_precision}
\end{figure}
\begin{table}[htbp]
\centering
\label{tb:ratio}
\begin{tabular}{|c|cc|ccc|}
\hline \multirow{2}{*}{ Generative Model } & \multicolumn{2}{c}{ Data Structure } & \multicolumn{3}{|c|}{ Average CPU time (s) } \\
& Agents(n) & Chores(m) & GFW & Ours(DCA) & Ours(SGR)\\
\hline Uniform & 500 & 50 & 2.01 & 0.38 & \textbf{0.05} \\
Uniform & 600 & 50 & 2.75 & 0.34 & \textbf{0.05} \\
 Uniform & 700 & 50 & 3.41 & 0.37 & \textbf{0.07} \\
 Uniform & 800 & 50 & 3.08 & 0.51 & \textbf{0.09} \\
 Uniform & 900 & 50 & 4.60 & 0.69 & \textbf{0.09} \\
 Uniform & 1000 & 50 & 5.63 & 0.82 & \textbf{0.11} \\
\hline
\end{tabular}
\caption{CPU Time Comparison under High Agent-Chore Ratio, $\epsilon=0.01$}
\end{table}


\section{Conclusion}\label{sec:conclusion}
In this paper, we developed a novel unconstrained DC formulation for computing CE for chores and proposed two algorithms based on the new formulation: DCA and SGR. Our algorithms address the limitations of existing methods. Specifically, the DCA achieves R-linear convergence to an exact CE, and the SGR for finding an approximate CE, which is subproblem-free, has a polynomial iteration complexity. The main ingredients for achieving these results are the K{\L} property (with exponent $1/2$) and the smoothing strategy. They closely rely on the structure of the DC formulation. We believe that these techniques may find their applications in other market equilibrium problems, yielding strong convergence results for lightweight methods.

\begin{credits}
\subsubsection{\ackname} Anthony Man-Cho So was supported in part by the Hong Kong Research Grants Council (RGC) General Research Fund (GRF) project CUHK 14204823.
\end{credits}

\bibliographystyle{splncs04}
\bibliography{ref}

\begin{thebibliography}{10}
\providecommand{\url}[1]{\texttt{#1}}
\providecommand{\urlprefix}{URL }
\providecommand{\doi}[1]{https://doi.org/#1}

\bibitem{arrow1954existence}
Arrow, K.J., Debreu, G.: Existence of an equilibrium for a competitive economy.
  Econometrica: Journal of the Econometric Society pp. 265--290 (1954)

\bibitem{attouch2013convergence}
Attouch, H., Bolte, J., Svaiter, B.F.: Convergence of descent methods for
  semi-algebraic and tame problems: Proximal algorithms, forward--backward
  splitting, and regularized {G}auss--{S}eidel methods. Mathematical
  Programming, Series A  \textbf{137}(1),  91--129 (2013)

\bibitem{balseiro2021regularized}
Balseiro, S., Lu, H., Mirrokni, V.: Regularized online allocation problems:
  {F}airness and beyond. In: International Conference on Machine Learning. pp.
  630--639. PMLR (2021)

\bibitem{beck2003mirror}
Beck, A., Teboulle, M.: Mirror descent and nonlinear projected subgradient
  methods for convex optimization. Operations Research Letters  \textbf{31}(3),
   167--175 (2003)

\bibitem{birnbaum2011distributed}
Birnbaum, B., Devanur, N.R., Xiao, L.: Distributed algorithms via gradient
  descent for {F}isher markets. In: Proceedings of the 12th ACM Conference on
  Electronic Commerce. pp. 127--136 (2011)

\bibitem{bogomolnaia2019dividing}
Bogomolnaia, A., Moulin, H., Sandomirskiy, F., Yanovskaia, E.: Dividing bads
  under additive utilities. Social Choice and Welfare  \textbf{52}(3),
  395--417 (2019)

\bibitem{bogomolnaia2017competitive}
Bogomolnaia, A., Moulin, H., Sandomirskiy, F., Yanovskaya, E.: Competitive
  division of a mixed manna. Econometrica  \textbf{85}(6),  1847--1871 (2017)

\bibitem{bolte2010characterizations}
Bolte, J., Daniilidis, A., Ley, O., Mazet, L.: Characterizations of
  {{\L}}ojasiewicz inequalities: Subgradient flows, talweg, convexity.
  Transactions of the American Mathematical Society  \textbf{362}(6),
  3319--3363 (2010)

\bibitem{boodaghians2022polynomial}
Boodaghians, S., Chaudhury, B.R., Mehta, R.: Polynomial time algorithms to find
  an approximate competitive equilibrium for chores. In: Proceedings of the
  2022 Annual ACM-SIAM Symposium on Discrete Algorithms (SODA). pp. 2285--2302.
  SIAM (2022)

\bibitem{boyd2004convex}
Boyd, S.P., Vandenberghe, L.: Convex {O}ptimization. Cambridge University Press
  (2004)

\bibitem{brainard2005compute}
Brainard, W.C., Scarf, H.E.: How to compute equilibrium prices in 1891.
  American Journal of Economics and Sociology  \textbf{64}(1),  57--83 (2005)

\bibitem{budish2011combinatorial}
Budish, E.: The combinatorial assignment problem: {A}pproximate competitive
  equilibrium from equal incomes. Journal of Political Economy
  \textbf{119}(6),  1061--1103 (2011)

\bibitem{chambolle2011first}
Chambolle, A., Pock, T.: A first-order primal-dual algorithm for convex
  problems with applications to imaging. Journal of Mathematical Imaging and
  Vision  \textbf{40},  120--145 (2011)

\bibitem{chaudhury2024competitive}
Chaudhury, B.R., Kroer, C., Mehta, R., Nan, T.: Competitive equilibrium for
  chores: {F}rom dual {E}isenberg-{G}ale to a fast, greedy, {LP}-based
  algorithm (2024)

\bibitem{chaudhury2018combinatorial}
Chaudhury, B.R., Mehlhorn, K.: Combinatorial algorithms for general linear
  {A}rrow-{D}ebreu markets. In: 38th IARCS Annual Conference on Foundations of
  Software Technology and Theoretical Computer Science (2018)

\bibitem{cole2017convex}
Cole, R., Devanur, N., Gkatzelis, V., Jain, K., Mai, T., Vazirani, V.V.,
  Yazdanbod, S.: Convex program duality, {F}isher markets, and {N}ash social
  welfare. In: Proceedings of the 2017 ACM Conference on Economics and
  Computation. pp. 459--460 (2017)

\bibitem{conitzer2022pacing}
Conitzer, V., Kroer, C., Panigrahi, D., Schrijvers, O., Stier-Moses, N.E.,
  Sodomka, E., Wilkens, C.A.: Pacing equilibrium in first price auction
  markets. Management Science  \textbf{68}(12),  8515--8535 (2022)

\bibitem{devanur2008market}
Devanur, N.R., Papadimitriou, C.H., Saberi, A., Vazirani, V.V.: Market
  equilibrium via a primal--dual algorithm for a convex program. Journal of the
  ACM (JACM)  \textbf{55}(5),  1--18 (2008)

\bibitem{duan2015combinatorial}
Duan, R., Mehlhorn, K.: A combinatorial polynomial algorithm for the linear
  {A}rrow--{D}ebreu market. Information and Computation  \textbf{243},
  112--132 (2015)

\bibitem{duchi2008efficient}
Duchi, J., Shalev-Shwartz, S., Singer, Y., Chandra, T.: Efficient projections
  onto the \(\ell_1\)-ball for learning in high dimensions. In: Proceedings of
  the 25th International Conference on Machine learning. pp. 272--279 (2008)

\bibitem{eisenberg1959consensus}
Eisenberg, E., Gale, D.: Consensus of subjective probabilities: {T}he
  pari-mutuel method. The Annals of Mathematical Statistics  \textbf{30}(1),
  165--168 (1959)

\bibitem{gale1989theory}
Gale, D.: The Theory of Linear Economic Models. University of Chicago Press
  (1989)

\bibitem{gao2020first}
Gao, Y., Kroer, C.: First-order methods for large-scale market equilibrium
  computation. Advances in Neural Information Processing Systems  \textbf{33},
  21738--21750 (2020)

\bibitem{garg2019strongly}
Garg, J., V{\'e}gh, L.A.: A strongly polynomial algorithm for linear exchange
  markets. In: Proceedings of the 51st Annual ACM SIGACT Symposium on Theory of
  Computing. pp. 54--65 (2019)

\bibitem{jain2007polynomial}
Jain, K.: A polynomial time algorithm for computing an {A}rrow--{D}ebreu market
  equilibrium for linear utilities. SIAM Journal on Computing  \textbf{37}(1),
  303--318 (2007)

\bibitem{jalota2023fisher}
Jalota, D., Pavone, M., Qi, Q., Ye, Y.: Fisher markets with linear constraints:
  {E}quilibrium properties and efficient distributed algorithms. Games and
  Economic Behavior  \textbf{141},  223--260 (2023)

\bibitem{jalota2022stochastic}
Jalota, D., Ye, Y.: Stochastic online {F}isher markets: {S}tatic pricing limits
  and adaptive enhancements. arXiv preprint arXiv:2205.00825  (2022)

\bibitem{li2018calculus}
Li, G., Pong, T.K.: Calculus of the exponent of {K}urdyka-{{\L}}ojasiewicz
  inequality and its applications to linear convergence of first-order methods.
  Foundations of Computational Mathematics  \textbf{18}(5),  1199--1232 (2018)

\bibitem{lin2020near}
Lin, T., Jin, C., Jordan, M.I.: Near-optimal algorithms for minimax
  optimization. In: Conference on Learning Theory. pp. 2738--2779. PMLR (2020)

\bibitem{liu2019quadratic}
Liu, H., So, A.M.C., Wu, W.: Quadratic optimization with orthogonality
  constraint: {E}xplicit {{\L}}ojasiewicz exponent and linear convergence of
  retraction-based line-search and stochastic variance-reduced gradient
  methods. Mathematical Programming, Series A  \textbf{178},  215--262 (2019)

\bibitem{luo1993error}
Luo, Z.Q., Tseng, P.: Error bounds and convergence analysis of feasible descent
  methods: {A} general approach. Annals of Operations Research  \textbf{46}(1),
   157--178 (1993)

\bibitem{nan2024convergence}
Nan, T., Gao, Y., Kroer, C.: On the convergence of {T}\^atonnement for linear
  {F}isher markets. arXiv preprint arXiv:2406.12526  (2024)

\bibitem{nesterov2005smooth}
Nesterov, Y.: Smooth minimization of non-smooth functions. Mathematical
  programming  \textbf{103},  127--152 (2005)

\bibitem{orlin2010improved}
Orlin, J.B.: Improved algorithms for computing {F}isher's market clearing
  prices: {C}omputing {F}isher's market clearing prices. In: Proceedings of the
  forty-second ACM Symposium on Theory of Computing. pp. 291--300 (2010)

\bibitem{peysakhovich2023implementing}
Peysakhovich, A., Kroer, C., Usunier, N.: Implementing fairness constraints in
  markets using taxes and subsidies. In: Proceedings of the 2023 ACM Conference
  on Fairness, Accountability, and Transparency. pp. 916--930 (2023)

\bibitem{rockafellar2015convex}
Rockafellar, R.T.: Convex Analysis:(PMS-28). Princeton University Press (2015)

\bibitem{schneider2015convergence}
Schneider, R., Uschmajew, A.: Convergence results for projected line-search
  methods on varieties of low-rank matrices via {{\L}}ojasiewicz inequality.
  SIAM Journal on Optimization  \textbf{25}(1),  622--646 (2015)

\bibitem{shmyrev2009algorithm}
Shmyrev, V.I.: An algorithm for finding equilibrium in the linear exchange
  model with fixed budgets. Journal of Applied and Industrial Mathematics
  \textbf{3},  505--518 (2009)

\bibitem{vavasis1996primal}
Vavasis, S.A., Ye, Y.: A primal-dual interior point method whose running time
  depends only on the constraint matrix. Mathematical Programming
  \textbf{74}(1),  79--120 (1996)

\bibitem{walras2014leon}
Walras, L.: L{\'e}on {W}alras: {E}lements of Theoretical Economics: Or, the
  Theory of Social Wealth. Cambridge University Press (2014)

\bibitem{wang2023linear}
Wang, P., Liu, H., So, A.M.C.: Linear convergence of a proximal alternating
  minimization method with extrapolation for \({\ell_1}\)-norm principal
  component analysis. SIAM Journal on Optimization  \textbf{33}(2),  684--712
  (2023)

\bibitem{yang2022auc}
Yang, T., Ying, Y.: Auc maximization in the era of big data and {AI}: {A}
  survey. ACM Computing Surveys  \textbf{55}(8),  1--37 (2022)

\bibitem{ye2008path}
Ye, Y.: A path to the {A}rrow--{D}ebreu competitive market equilibrium.
  Mathematical Programming  \textbf{111}(1),  315--348 (2008)

\bibitem{zhao2023primal}
Zhao, R.: A primal-dual smoothing framework for max-structured non-convex
  optimization. Mathematics of Operations Research  (2023)

\bibitem{zhou2017unified}
Zhou, Z., So, A.M.C.: A unified approach to error bounds for structured convex
  optimization problems. Mathematical Programming, Series A  \textbf{165}(2),
  689--728 (2017)

\end{thebibliography}
\newpage
\appendix
The appendix is structured as follows. 
In Appendix \ref{appen:kl}, we prove the error bound condition and K{\L} exponent property with exponent $1/2$ of the new formulation \eqref{eq:DC}.  Appendix \ref{appen:dca} and \ref{appen:sgr} provide missing proofs for our DCA and SGR, respectively.  In Appendix \ref{appen:round_dca}, we present a rounded DCA that possesses a non-asymptotic rate for finding an approximate CE. We provide experiment details in Appendix \ref{appen:additional_exp}.
 
\section{Proofs of Error Bound and K{\L} exponent}\label{appen:kl}
\subsection{Proof of Proposition \ref{pro:EB}}
    Let us first recall the notation 
    \[h_i(\mu)=B_i\max_{j\in[m]}\{\mu_j-\log(d_{ij})\}; \quad \ell(\mu)=\sum_{i\in[n]}B_i\log\left(\sum_{j\in[m]}e^{\mu_j}\right);\quad J_i(\mu)=\{j:\mu_j-\log(d_{ij})=h_i(\mu)\}.\]
    Given the base point $\bar{\mu}$, we classify the points $\mu\in\R^m$ into two types. We say that:
\begin{enumerate}
    \item $\mu$ is a type I point if for all $v_i\in \partial h_i(\mu)$, it holds that 
\begin{equation}\label{eq:typeI}
\sum_{i\in[n]}v_i-\nabla \ell(\bar{\mu})\neq0;\tag{Type I}
\end{equation}
\item  $\mu$ is a type II point if there exist $\hat{v}_i\in \partial h_i(\mu) $ such that
\begin{equation}\label{eq:typeII}
\sum_{i\in[n]}\hat{v}_i-\nabla \ell(\bar{\mu})=0.\tag{Type II}
\end{equation}
\end{enumerate}
 We prove that type I and II points satisfy the local error bound, respectively.
\paragraph{Type I:} For a type I point $\mu$, let $\cJ(\mu)$ denote the Cartesian product of $J_i(\mu)$, i.e., \[\cJ(\mu)=J_1(\mu)\times J_2(\mu)\times\cdots \times J_n(\mu).\]
Using the expression of $\partial h_i(\mu)$, we see 
\[\partial h_i(\mu)=\left\{v_i:\sum_{j\in[m]}v_{ij}=B_i,\ v_{ij}\geq0,\ v_{ij}=0\text{ if }j\notin J_i(\mu)\right\},\]which is a polyhedron. 
Then, the set $V_{\cJ(\mu)}\coloneqq\sum_{i\in[n]}\partial h_i(\mu)$
is also a polyhedron by \cite[Corollary 19.3.2]{rockafellar2015convex}, which only depends on $\cJ(\mu)$. Now, \eqref{eq:typeI} is equivalent to
\[\nabla\ell(\bar{\mu})\notin V_{\cJ(\mu)},\text{ i.e., }\dist\left(\nabla\ell(\bar{\mu}),V_{\cJ(\mu)}\right)>0.\]
Since $\cJ(\mu)\subseteq \Xi\coloneqq[m]\times\cdots\times [m]$ belongs to a finite discrete set, we know 
\[\delta_1\coloneqq \min\limits_{\cJ(\mu)\subseteq \Xi}\left\{\dist\left(\nabla\ell(\bar{\mu}),V_{\cJ(\mu)}\right):\dist\left(\nabla\ell(\bar{\mu}),V_{\cJ(\mu)}\right)>0\right\}>0.\]
We are ready for the local error bound for type I points. For all type I points $\mu$ with $\|\mu-\bar{\mu}\|\leq\delta_1/{(2\sum_{i\in[n]}B_i)}$, we have the following estimates on $\dist(0,\partial F(\mu))$:
\begin{equation}\label{eq:dist_F1}
\begin{aligned}
    \dist\left(0,\partial F(\mu)\right)&=\dist\left(0,\sum\limits_{i\in[n]}\partial h_i(\mu)-\nabla\ell(\mu)\right)\\
    &=\dist\left(0,\sum\limits_{i\in[n]}\partial h_i(\mu)-\nabla\ell(\bar{\mu})+\nabla\ell(\bar{\mu})-\nabla\ell(\mu)\right)\\
    &\geq \dist\left(0,\sum\limits_{i\in[n]}\partial h_i(\mu)-\nabla\ell(\bar{\mu})\right)-\|\nabla\ell(\bar{\mu})-\nabla\ell(\mu)\|\\
    &\geq\dist\left(\nabla\ell(\bar{\mu}),V_{\cJ(\mu)}\right)-\sum_{i\in[n]}B_i\|\mu-\bar{\mu}\|\\
    & \geq \delta_1-\sum_{i\in[n]}B_i\frac{\delta_1}{2\sum_{i\in[n]}B_i}=\frac{\delta_1}{2},
\end{aligned}
\end{equation}
where the first inequality uses the triangle inequality, and the second one is due to the $\sum_{i\in[n]}B_i$-Lipschitz continuity of $\nabla\ell$ and definition of $V_{\cJ(\mu)}$. On the other hand, for type I points $\mu \text{ with }\|\mu-\bar{\mu}\|\leq \frac{\delta_1}{2\sum_{i\in[n]}B_i}$, it holds that
\[\frac{\delta_1}{2}\geq \sum_{i\in[n]}B_i\|\mu-\bar{\mu}\|\geq \sum_{i\in[n]}B_i\dist(\mu,\cU^*).  \]
This, together with \eqref{eq:dist_F1}, implies that
\begin{equation}\label{eq:typeIerror}
   \dist(\mu,\cU^*)\leq \frac1{\sum_{i\in[n]}B_i}  \dist\left(0,\partial F(\mu)\right) \text{ for type I points }\mu \text{ with }\|\mu-\bar{\mu}\|\leq \frac{\delta_1}{2\sum_{i\in[n]}B_i}.
\end{equation}
Inequality \eqref{eq:typeIerror} is the desired local error bound for type I points with $\tau_1\coloneqq 1/\sum_{i\in[n]}B_i$.
\paragraph{Type II:} The proof of local error bound for type II points relies on a careful analysis of the maximum index sets $J_i(\mu)$. Specifically, we define classes for $J_i(\mu)$, $i\in[n]$, depending on whether they intersect (directly or indirectly), and then union $J_i(\mu)$ in the same class to get $\tilde{J}_l(\mu)$, $l\in[s]$. Then, the new index sets $\tilde{J}_l(\mu)$, $l\in[s]$ is a partition for $[m]$, and $\mu_j-\bar{\mu}_j$ equal for $j\in \tilde{J}_l(\mu)$ due to the following fact. This facilitates a decomposition of $\partial F(\mu)$, which is key to error bound for type II points.
    \begin{Fact}\label{fact:J}
        There exists $\delta_2>0$ such that for all $\mu$ with $\|\mu-\bar{\mu}\|\leq \delta_2$, it holds that $J_i(\mu)\subseteq J_i(\bar{\mu})$. Furthermore, for each $i\in[n]$, $\mu_{j}-\bar{\mu}_{j}$ equal for all $j\in J_i(\mu)$.
    \end{Fact}
    \begin{lemma}\label{le:mu}
       Given a type II point $\mu$ with $\|\mu-\bar{\mu}\|\leq \delta_2$, there are partitions $[n]=\bigcup_{l\in[s]}I_l(\mu)$ and $[m]=\bigcup_{l\in[s]}\tilde{J}_l(\mu)$ that are determined by $\cJ(\mu)$ such that:
       \begin{enumerate}[{\rm (i)}]
       \item $\tilde{J}_l(\mu)=\bigcup_{i\in I_l}J_i(\mu)$;
           \item  For each $l\in[s]$, $\mu_j-\bar{\mu}_j$ equal for $j\in \tilde{J}_l(\mu)$;
           \item $I_l(\mu)\bigcap I_{l^{\prime}}(\mu)=\emptyset$ and $\tilde{J}_l(\mu)\bigcap \tilde{J}_{l^{\prime}}(\mu)=\emptyset$ if $l\neq l^{\prime}$.
       \end{enumerate}
    \end{lemma}
    \begin{proof}[Proof of Fact \ref{fact:J}]
        Let us set 
    \begin{equation}\label{eq:delta}
        \delta_2\coloneqq\frac13\min_{i\in[n]}\left\{ h_i(\bar{\mu})-\max_{j\in[m]\setminus J_i(\bar{\mu})}\left\{\mu_j-\log(d_{ij})\right\}\right\}.
    \end{equation}
    Then $\delta_2>0$ and for all $\|\mu-\mu^*\|\leq\delta_2$, it holds that
    \begin{equation}\label{eq:index}
        \begin{aligned}
  \max_{j\in J_i(\bar{\mu})}\left\{\mu_j-\log(d_{ij})\right\}&\geq \ \  \max_{j\in J_i(\bar{\mu})}\left\{\bar{\mu}_j-\log(d_{ij})\right\}-\|\mu-\bar{\mu}\|\\
  &\geq\quad h_i(\bar{\mu})-\delta_2\\
  &\geq \max_{j\in[m]\setminus J_i(\bar{\mu})}\left\{\bar{\mu}_j-\log(d_{ij})\right\}+3\delta_2-\delta_2\\
  & \geq \max_{j\in[m]\setminus J_i(\bar{\mu})}\left\{{\mu}_j-\log(d_{ij})\right\}-\|\mu-\bar{\mu}_j\|+2\delta_2\\
  & = \max_{j\in[m]\setminus J_i(\bar{\mu})}\left\{{\mu}_j-\log(d_{ij})\right\}+\delta_2,\\
    \end{aligned}
\end{equation}
where the first and fourth inequalities use the $1$-Lipschitz continuity of function $\mu\mapsto \max_{j\in J}\{\mu_j-\log(d_{ij})\}$ (with $J\subseteq[m]$ being an index set), and the third inequality is due to the definition \eqref{eq:delta}.
Using \eqref{eq:index}, together with the definition of $J_i(\mu)$ and $\delta_2>0$, we see that $J_i(\mu)\subseteq J_i(\bar{\mu})$. 

Finally, we show that $\mu_{j}-\bar{\mu}_{j}$ equal for all $j\in J_i(\mu)$. It suffices to notice that $J_i(\mu)\subseteq J_i(\bar{\mu})$ ensures $h_i(\bar{\mu})=\bar{\mu}_j-\log(d_{ij})$ for $j\in J_i({\mu})$, and thus
\[h_i(\mu)-h_i(\bar{\mu})=\left({\mu}_j-\log(d_{ij})\right)-\left(\bar{\mu}_j-\log(d_{ij})\right)=\mu_j-\bar{\mu}_j\text{ for }j\in J_i(\mu).\]
    \end{proof}
    \begin{proof}[Proof of Lemma \ref{le:mu}] The starting point is to partition $[m]$. Note that $\mu$ is a type II point, i.e., $\eqref{eq:typeII}$ holds, and $\nabla_{j}\ell(\bar{\mu})>0$ for each $j\in[m]$ by direct computation. We see that for each  $j\in[m]$, there is an index $i\in[n]$ such that $\hat{v}_{ij}>0$, and further $j\in J_i(\mu)$ due to $\hat{v}_i\in\partial h_i(\mu)$ and $h_i(\mu)=\conv\{B_ie_j:j\in J_i(\mu)\}$. It follows that $[m]\subseteq \bigcup_{i\in[n]}J_{i}(\mu)$. Since $J_{i}(\mu)\subseteq[m]$, we have $[m]=\bigcup_{i\in[n]}J_{i}(\mu)$.
    
    Then, to partition $[m]$, it suffices to partition $\{J_i(\mu):i\in[n]\}$ into different equivalence classes. Specifically, we write $J_{i}(\mu)\sim J_{i^{\prime}}(\mu)$ if $J_{i}(\mu)\bigcap J_{i^{\prime}}(\mu)\neq \emptyset$ or  there exist $i_1,i_2,\ldots,i_l\in[m]$ such that
    \[J_{i}(\mu)\bigcap J_{i_1}(\mu)\neq\emptyset,\ J_{i_1}(\mu)\bigcap J_{i_2}(\mu)\neq\emptyset,\ldots, J_{i_{l-1}}(\mu)\bigcap J_{i_l}(\mu)\neq\emptyset,\ J_{i_l}(\mu)\bigcap J_{i^{\prime}}(\mu)\neq\emptyset. \]
    This definition ensures: $J_{i}(\mu)\sim J_{i}(\mu)$; if $J_{i}(\mu)\sim J_{i^{\prime}}(\mu)$, then $ J_{i^{\prime}}(\mu)\sim J_{i}(\mu)$; if $J_{i_1}(\mu)\sim J_{i_2}(\mu)$ and $J_{i_2}(\mu)\sim J_{i_3}(\mu)$, then $J_{i_1}(\mu)\sim J_{i_3}(\mu)$. So the relation "$\sim$" defines equivalence classes for $\{J_i(\mu):i\in[n]\}$. 
    
    Assume without loss of generality (WLOG for short) that there are $s$ equivalence classes and let $I_l(\mu),\ l\in[s]$ contain the indices $i$ of each equivalence class. It follows from the definition of these equivalence classes that
    \begin{equation}\label{eq:Ji}
       I_l(\mu)\bigcap I_{l^{\prime}}(\mu)=\emptyset\text{ and } J_i(\mu)\bigcap J_{i^{\prime}}(\mu)=\emptyset \text{ if }i\in I_{l},\ i^{\prime}\in I_{l^{\prime}}\text{ with }l\neq l^{\prime}.
    \end{equation}
    Define $\tilde{J}_l(\mu)\coloneqq \bigcup_{i\in I_l}J_i(\mu)$. Then, Item (i) automatically holds and \eqref{eq:Ji} implies Item (iii).
    
    To prove Item (ii), we recall Fact \ref{fact:J} that ensures that for each $i\in[n]$, $\mu_j-\bar{\mu}_j$ equal for $j\in J_i(\mu)$. We see that if $J_i(\mu)\bigcap J_{i^{\prime}}(\mu)\neq \emptyset$, then $\mu_j-\bar{\mu}_j$ equal for $j\in J_i(\mu)\bigcup J_{i^{\prime}}(\mu)$. This, together with the definition of our equivalence classes, implies that for each $l\in[s]$,
        \[\mu_j-\bar{\mu}_j\text{ equal for all }j\in\bigcup\limits_{i\in I_l}J_i(\mu)=\tilde{J}_l(\mu). \]
        This completes the proof of Lemma \ref{le:mu}.
    \end{proof}
    The left proof for the local error bound of type II points consists of two parts.
    
\textbf{(i) Decomposition-based estimation for} $\dist(0,\partial F(\mu))$: Let us consider following equivalent form of $\dist^2(0,\partial F(\mu))$:
\begin{equation}\label{eq:distF2}
    \begin{aligned}
    \dist^2\left(0,\partial F(\mu)\right)&=\dist^2\left(0,\sum\limits_{i\in[n]}\partial h_i(\mu)-\nabla\ell(\mu)\right)\\
    &=\dist^2\left(0,\sum\limits_{i\in[n]}\partial h_i(\mu)-\nabla\ell(\bar{\mu})+\nabla\ell(\bar{\mu})-\nabla\ell(\mu)\right)\\
    &=\dist^2\left(\nabla\ell({\mu})-\nabla\ell(\bar{\mu}),\sum\limits_{i\in[n]}\partial h_i(\mu)-\nabla\ell(\bar{\mu})\right)\\
    &=\dist^2\left(\nabla\ell({\mu})-\nabla\ell(\bar{\mu}),\sum\limits_{l\in[s]}\sum\limits_{i\in I_l(\mu)}\left(\conv\{B_ie_j:j\in J_i(\mu)\}-\hat{v}_i \right)\right)\\
\end{aligned}
\end{equation}
where the last equality uses \eqref{eq:typeII} and the partition $[n]=\bigcup_{l\in[s]}I_l(\mu)$ given by Lemma \ref{le:mu}. Recall that $\hat{v}_i\in \partial h_i(\mu)=\conv\{B_ie_j:j\in J_i(\mu)\}$. We see that for $i\in I_l(\mu)$, \[\begin{aligned}
    \conv\{B_ie_j:j\in J_i(\mu)\}-\hat{v}_i &\subseteq \left\{\lambda\in\R^m:\sum_{j\in J_i(\mu)}\lambda_j=0,\ \lambda_j=0\text{ for }j\notin J_i(\mu)\right\}\\
    &\subseteq \cC_l(\mu)\coloneqq\left\{\lambda\in\R^m:\sum_{j\in \tilde{J}_l(\mu)}\lambda_j=0,\ \lambda_j=0\text{ for }j\notin \tilde{J}_l(\mu)\right\}\\
\end{aligned}.\]
Observe that by definition of $\cC_l(\mu)$, the sum of subsets of $\cC_l(\mu)$ is still its subset. It follows that
\begin{equation}\label{eq:cC}
    \sum\limits_{i\in I_l(\mu)}\left(\conv\{B_ie_j:j\in J_i(\mu)\}-\hat{v}_i \right)\subseteq \cC_l(\mu).
\end{equation}
Combining \eqref{eq:distF2} and \eqref{eq:cC}, we have
\begin{equation}\label{eq:distF3}
    \dist^2\left(0,\partial F(\mu)\right)\geq \dist^2\left(\nabla\ell({\mu})-\nabla\ell(\bar{\mu}),\sum\limits_{l\in[s]}\cC_l(\mu)\right). 
\end{equation}
We write 
\[u_l(\mu)\coloneqq \Pi_{{\rm span}\{e_j:j\in \tilde{J}_l(\mu)\}}\left(\nabla\ell({\mu})-\nabla\ell(\bar{\mu})\right),\]
and have the decomposition $\nabla\ell({\mu})-\nabla\ell(\bar{\mu})=\sum_{l\in[s]}u_l(\mu)$ due to the partition $[m]=\bigcup_{l\in[s]}\tilde{J}_l(\mu)$, where $\tilde{J}_l(\mu)\cap \tilde{J}_{l^{\prime}}(\mu)=\emptyset$ if $l\neq l^{\prime}$. Then, the right-hand side of \eqref{eq:distF3} can be further transformed:
\begin{equation}\label{eq:distF4}
    \begin{aligned}
        \dist^2\left(\nabla\ell({\mu})-\nabla\ell(\bar{\mu}),\sum\limits_{l\in[s]}\cC_l(\mu)\right)&=\min\limits_{\lambda_l\in\cC_l(\mu),l\in[s]}\quad \left\|\sum\limits_{l\in[s]}\left(u_l(\mu)-\lambda_l\right) \right\|^2\\
        &=\sum\limits_{l\in[s]}\min\limits_{\lambda_l\in\cC_l(\mu)}\quad\left\|  u_l(\mu)-\lambda_l\right\|^2\\
        &=\sum\limits_{l\in[s]}\dist^2\left(u_l(\mu),\cC_l(\mu) \right).
    \end{aligned}
\end{equation}
Here, the second step is due to the orthogonality of $\{u_l(\mu)-\lambda_l:l\in[s]\}$. This can be seen from the facts that $u_l(\mu)\in{\rm span}\{e_j:j\in \tilde{J}_l(\mu)\}$, $\lambda_l\in\cC_l(\mu)\subseteq {\rm span}\{e_j:j\in \tilde{J}_l(\mu)\}$, and ${\rm span}\{e_j:j\in \tilde{J}_l(\mu)\},\ l\in[s]$ are orthogonal subspaces of $\R^m$ due to Lemma \ref{le:mu} (iii).
We define
\[\cD_{\left|\tilde{J}_l(\mu)\right|}\coloneqq \left\{\lambda\in\R^{\left|\tilde{J}_l(\mu)\right|}:\sum\limits_{j\in\left[\left|\tilde{J}_l(\mu)\right|\right]}\lambda_j=0\right\}.\]
By the definitions of $u_l(\mu)$ and $\cC_l(\mu)$, we have
\begin{equation}\label{eq:dist_ell}
    \begin{aligned}
        \dist^2\left(u_l(\mu),\cC_l(\mu) \right)= \dist^2\left(\nabla_{\tilde{J}_l(\mu)}\ell({\mu})-\nabla_{\tilde{J}_l(\mu)}\ell(\bar{\mu}),\cD_{\left|\tilde{J}_l(\mu)\right|}\right).
    \end{aligned}
\end{equation}
Combining \eqref{eq:distF3}, \eqref{eq:distF4}, and \eqref{eq:dist_ell}, we have
\begin{equation}\label{eq:trans}
     \dist^2\left(0,\partial F(\mu)\right)\geq\sum\limits_{l\in[s]}\dist^2\left(\nabla_{\tilde{J}_l(\mu)}\ell({\mu})-\nabla_{\tilde{J}_l(\mu)}\ell(\bar{\mu}),\cD_{\left|\tilde{J}_l(\mu)\right|}\right).
\end{equation}
\textbf{(ii) Estimation on subspace:}
Let us then estimate the right-hand side of \eqref{eq:dist_ell}. Construct the following stationary point $\mu^*$ associated with $\mu$ that satisfies $\sum_{j\in[m]} e^{\mu^*_j}= \sum_{j\in[m]} e^{\mu_j}$:
\begin{equation}\label{eq:mu*}
    \mu^*\coloneqq \bar{\mu}+\log\left({\sum_{j\in[m]}e^{\mu_j} }\right)\1-\log\left({\sum_{j\in[m]}e^{\bar{\mu}_j} }\right)\1.
\end{equation}
Here , the constructed $\mu^*$ is stationary for \eqref{eq:DC} due to Fact \ref{fact:homo}. Furthermore, we see that $\nabla \ell(\mu^*)=\nabla\ell(\bar{\mu})$ by direct computation and $\mu_j-\mu^*_j$ equal for $j\in\tilde{J}_l(\mu)$ by  Lemma \ref{le:mu} (ii). Let $a_l(\mu)\coloneqq \mu_j-\mu^*_j$ for $j\in\tilde{J}_l(\mu)$. Using $\sum_{j\in[m]} e^{\mu^*_j}= \sum_{j\in[m]} e^{\mu_j}$, we have
\begin{equation}\label{eq:sub1}
    \begin{aligned}
        \dist\left(\nabla_{\tilde{J}_l(\mu)}\ell({\mu})-\nabla_{\tilde{J}_l(\mu)}\ell(\bar{\mu}),\cD_{\left|\tilde{J}_l(\mu)\right|}\right)&=\dist\left(\nabla_{\tilde{J}_l(\mu)}\ell({\mu})-\nabla_{\tilde{J}_l(\mu)}\ell({\mu}^*),\cD_{\left|\tilde{J}_l(\mu)\right|}\right)\\
        &=\dist\left(\frac{\sum_{i\in[n]}B_i\left(e^{\mu_j}-e^{\mu^*_j}\right)_{\tilde{J}_l(\mu)} }{\sum_{j\in[m]} e^{\mu_j}},\cD_{\left|\tilde{J}_l(\mu)\right|}\right)\\
        &=\frac{\sum_{i\in[n]}B_i\left(e^{\mu_j}-e^{\mu^*_j}\right)^{\top}_{\tilde{J}_l(\mu)} }{\sum_{j\in[m]} e^{\mu_j}} \frac{\1}{\sqrt{\left|\tilde{J}_l(\mu)\right|}}\\
        &=\sum_{i\in[n]}B_i\frac{\left|1-e^{-a_l(\mu)}\right|\sum\limits_{j\in\tilde{J}_l(\mu)}e^{\mu_j}}{\sqrt{\left|\tilde{J}_l(\mu)\right|}\sum\limits_{j\in[m]} e^{\mu_j}}.
    \end{aligned}
\end{equation}
We then estimate the terms on the right-hand side of \eqref{eq:sub1}.
For $\mu$ with $\|\mu-\bar{\mu}\|\leq\delta_2$, 
\[\sum\limits_{j\in[m]} e^{\mu_j}\leq \sum\limits_{j\in[m]} e^{\bar{\mu}_j+\delta_2}\leq me^{\max_{j}\bar{\mu}_j+\delta_2}\text{; }\sum\limits_{j\in\tilde{J}_l(\mu)}e^{\mu_j}\geq \sum\limits_{j\in\tilde{J}_l(\mu)}e^{\bar{\mu}_j-\delta_2}\geq \left|\tilde{J}_l(\mu)\right|e^{\min_{j}\bar{\mu}_j-\delta_2}, \]
and for $l\in[s],\ j\in \tilde{J}_l(\mu)$,
\[|a_l(\mu)|=|\mu_j-\bar{\mu}_j+\bar{\mu}_j-\mu^*_j|\leq|\mu_j-\bar{\mu}_j|+|\bar{\mu}_j-\mu^*_j|\leq \delta_2+\left|\log\left(\frac{\sum_{j\in[m]} e^{\bar{\mu}_j}}{\sum_{j\in[m]} e^{{\mu}_j}}\right) \right|\leq 2\delta_2.\]
The latter further implies that for $l\in[s],\ j\in \tilde{J}_l(\mu)$,
\[|1-e^{-a_l(\mu)}|\geq e^{-|a_l(\mu)|}|a_l(\mu)|\geq e^{-2\delta_2}|\mu_j-\mu_j^*|= \frac{e^{-2\delta_2}}{\sqrt{\left|\tilde{J}_l(\mu)\right|}}\left\|\mu_{\tilde{J}_l(\mu)}-\mu^*_{\tilde{J}_l(\mu)}\right\|.\]
These estimates, together with \eqref{eq:sub1}, imply that
\begin{equation}\label{eq:sub2}
      \dist\left(\nabla_{\tilde{J}_l(\mu)}\ell({\mu})-\nabla_{\tilde{J}_l(\mu)}\ell(\bar{\mu}),\cD_{\left|\tilde{J}_l(\mu)\right|}\right)\geq  \frac{ e^{-4\delta_2}}{me^{\max_j\bar{\mu}_j-\min_{j}\bar{\mu}_j}}\left\|\mu_{\tilde{J}_l(\mu)}-\mu^*_{\tilde{J}_l(\mu)}\right\|.
\end{equation}
Let $\tau_2\coloneqq {(me^{\max_j\bar{\mu}_j-\min_{j}\bar{\mu}_j})}/e^{-4\delta_2}$. Then \eqref{eq:trans} and \eqref{eq:sub2} imply that for type II points $\mu$ with $\|\mu-\bar{\mu}\|\leq\delta_2$,
\begin{equation}\label{eq:dist_square}
    \dist^2\left(0,\partial F(\mu)\right)\geq\sum\limits_{l\in[s]}\frac{1}{\tau_2^2}\left\|\mu_{\tilde{J}_l(\mu)}-\mu^*_{\tilde{J}_l(\mu)}\right\|^2=\frac{1}{\tau_2^2}\left\|\mu-\mu^*\right\|^2.
\end{equation}
Since $\mu^*$ is stationary, we have $\|\mu-\mu^*\|\geq\dist(\mu,\cU^*)$, and then the error bound follows:
\begin{equation}\label{eq:typeIIerror}
    \dist\left(\mu,\cU^*\right)\leq\tau_2\dist\left(0,\partial F(\mu)\right).
\end{equation}
Finally, set $\delta\coloneqq\min\{\delta_1/(2\sum_{i\in[n]}B_i),\delta_2\}$ and $\tau\coloneqq\max\{\tau_1,\tau_2\}$. The desired error bound for all $\mu$ with $\|\mu-\bar{\mu}\|\leq\delta$ is ensured by inequalities \eqref{eq:typeIerror} and \eqref{eq:typeIIerror}.
\subsection{Proof of Theorem \ref{th:KL}}
It suffices to consider the K{\L} exponents of stationary points due to \cite[Lemma 2.1]{li2018calculus}. 
Similar to the proof of Proposition \ref{pro:EB}, we separate points $\mu$ into types I and II, and prove the K{\L} property with exponent $1/2$ respectively.

\paragraph{Type I:} We revisit \eqref{eq:dist_F1}, which says that there exists a scalar $\delta_1>0$ such that for all type I points $\mu$ with $\|\mu-\bar{\mu}\|\leq\delta_1/{(2\sum_{i\in[n]}B_i)}$, 
\begin{equation}\label{eq:distdelta1}
    \dist\left(0,\partial F(\mu)\right)\geq\frac{\delta_1}{2}. 
\end{equation}
Noticing the $2\sum_{i\in[n]}B_i$-Lipschitz continuity of $F$, for type I points $\mu$ with $\|\mu-\bar{\mu}\|\leq\delta_1/(2\sum_{i\in[n]}B_i)$, we have
\begin{equation}\label{eq:typeIKL}
    F(\mu)-F(\bar{\mu})\leq 2\sum_{i\in[n]}B_i\|\mu-\bar{\mu}\|\leq  \delta_1\leq \frac{4}{\delta_1} \dist^2\left(0,\partial F(\mu)\right),
\end{equation}
where the last inequality uses \eqref{eq:distdelta1}.
\paragraph{Type II:} 
By Fact \ref{fact:J}, there exists a scalar $\delta_2>0$ such that for all $\mu$ with $\|\mu-\bar{\mu}\|\leq \delta_2$, it holds that $J_i(\mu)\subseteq J_i(\bar{\mu})$. Adopt the definition \eqref{eq:mu*}. Then $J_i(\bar{\mu})=J_i(\mu^*)\supseteq J_i(\mu)$.
Subsequently, for a type II point $\mu$ and $j\in J_i(\mu)$, we have
\[h_i(\mu)-h_i({\mu^*})=B_i(\mu_j-\mu_j^{*})=(\mu-{\mu^*})^{\top}\hat{v}_i,\]
where the second equality is due to $\hat{v}_i\in\partial h_i(\mu)=\conv\{B_ie_j:j\in J_i(\mu)\}$ and that $\mu_j-\mu_j^{*}$ equal for $j\in J(\mu)$. It follows that
\begin{equation}\label{eq:nonsmooth}
    \sum\limits_{i\in[n]}h_i(\mu)-  \sum\limits_{i\in[n]}h_i({\mu^*})=  \sum\limits_{i\in[n]}(\mu-{\mu^*})^{\top}\hat{v}_i.
\end{equation}
On the other hand, since $\ell$ is a $\sum_{i\in[n]}B_i$-smooth function, it holds that
\begin{equation*}
    -\ell(\mu)+\ell\left(\mu^*\right)\leq -\nabla\ell(\mu^*)^{\top}(\mu-\mu^*)+\frac{\sum_{i\in[n]}B_i}{2}\|\mu-\mu^*\|^2.
\end{equation*}
This, together with \eqref{eq:dist_square}, yields that
\begin{equation}\label{eq:smooth}
     -\ell(\mu)+\ell\left(\mu^*\right)\leq -\nabla\ell(\mu^*)^{\top}(\mu-\mu^*)+\frac{\sum_{i\in[n]}B_i\tau_2^2}{2}\dist^2\left(0,\partial F(\mu)\right)\text{ if }\|\mu-\bar{\mu}\|\leq\delta_2.
\end{equation}
 Summing \eqref{eq:nonsmooth} and \eqref{eq:smooth}, we see that for $\mu$ with $\|\mu-\bar{\mu}\|\leq\delta_2$,
 \begin{equation}\label{eq:F_F*}
     F(\mu)-F(\mu^*)\leq \left(\sum\limits_{i\in[n]}\hat{v}_i-\nabla\ell(\mu^*)\right)^{\top}(\mu-\mu^*)+\frac{\sum_{i\in[n]}B_i\tau_2^2}{2}\dist^2\left(0,\partial F(\mu)\right).
 \end{equation}
 Recall that the definition of $\mu^*$, i.e., \eqref{eq:mu*}, implies that $F(\mu^*)=F(\bar{\mu})$ and $\nabla\ell(\mu^*)=\nabla\ell(\bar{\mu})$. These, together with \eqref{eq:F_F*}, imply that for all type II points $\mu$ with $\|\mu-\bar{\mu}\|\leq\delta_2$, it holds that $\sum_{i\in[n]}\hat{v}_i-\nabla\ell(\mu^*)=0$ and further
 \begin{equation}\label{eq:typeIIKL}
F(\mu)-F(\bar{\mu})\leq \frac{\sum_{i\in[n]}B_i\tau_2^2}{2}\dist^2\left(0,\partial F(\mu)\right).  
 \end{equation}
Set $\epsilon=\min\{\delta_1/(2\sum_{i\in[n]}B_i),\delta_2\}$, $\eta=\max\{4/\delta_1,\sum_{i\in[n]}B_i\tau_2^2/2\}$, and $\nu\in(0,\infty)$. Then, \eqref{eq:typeIKL} and \eqref{eq:typeIIKL} prove the K{\L} property  with exponent $1/2$ at $\bar{\mu}$ with the objects $\epsilon,\eta,$ and $\nu$.
\section{Missing Proofs for DCA}\label{appen:dca}
We first introduce a useful lower bound $\underline{\mu}$, which plays a crucial role in analyzing the DCA iterates, stationary points of \eqref{eq:DC}, and bounds on the function $F$. 
\begin{Fact}\label{fact:underlined}
    For $\mu\in\R^m$ satisfying $\sum_{i\in[n]}e^{\mu_j}\geq\sum_{i\in[n]}B_i$, if $\mu_j<\underline{\mu}$ for some $j\in[m]$, where $\underline{\mu}\coloneqq \log (\sum_{i\in[n]}B_i/m )-\max_{i,j}\log(d_{ij})+\min_{i,j}\log(d_{ij})$, then for $i\in[n]$, all vectors $v_i\in\partial h_i(\mu)$ satisfy $v_{ij}=0$. Furthermore, each vector $u\in\partial F(\mu)$ satisfies $u_j<0$.
\end{Fact}
\begin{proof}
  By $\sum_{j\in[m]}e^{\mu_j}\geq\sum_{i\in[n]}B_i$, we have $\max_j{\mu_j}\geq\log(\sum_{i\in[n]}B_i/m)$. It follows that: If $\mu_j-\min_{i,j}\log(d_{ij})<\log(\sum_{i\in[n]}B_i/m)-\max_{i,j}\log(d_{ij})$, then $j\notin J_i(\mu)$ for all $i\in[n]$. Consequently, for all $\mu$ with $\sum_{j\in[m]}e^{\mu_j}\geq\sum_{i\in[n]}B_i$, 
    \begin{equation}\label{eq:notinJ}
     j\notin J_i(\mu)\text{ for all }   i\in[n] \text{ if }\mu_j<\underline{\mu}.
    \end{equation}
    
    If $\mu_j<\underline{\mu}$, then by \eqref{eq:notinJ},  we have $v_i\in\partial h_i(\mu)=0$ satisfying $v_{ij}=0$ for all $i\in[n]$. Since $\nabla\ell>0$ always holds on $\R^m$ and $\partial F=\sum_{i\in[n]}\partial h_i-\nabla\ell$, we see that $u_j<0$ for all $u\in \partial F(\mu)$, which completes the proof.
\end{proof}
Fact \ref{fact:underlined} yields $\mu_j\geq\underline{\mu}$ for all stationary $\mu$ with $\sum_{j\in[m]}e^{\mu_j}=\sum_{i\in[n]}B_i$.
Let $q_j(\mu)\coloneqq \sum_{i\in[n]}B_ie^{\mu_j}/\sum_{j\in[m]}e^{\mu_j}$ for $j\in[m]$. Notice that $q_j(\mu)=e^{\mu_j}$ for $\mu$ with $\sum_{j\in[m]}e^{\mu_j}=\sum_{i\in[n]}B_i$ and $q_j(\mu+t\1)=q_j(\mu)$ for all $t\in\R$. We have the following estimate for all stationary points of $F$.
\begin{corollary}\label{co:mu}
    If $\mu$ is a stationary point of \eqref{eq:DC}, then for all $j\in[m]$
    \[q_j(\mu)\geq e^{\underline{\mu}}=\frac{\sum_{i\in[n]}B_i\min_{i,j}d_{ij}}{m\max_{i,j}d_{ij}}.\]
\end{corollary}
Furthermore, we can utilize the lower bound $\underline{\mu}$ to give bounds on $\inf_{\mu}F(\mu)$ and $\sup_{\mu}F(\mu)$.
\begin{Fact}[Lower \& Upper Bounds on $F$]\label{fact:F}
    \[\inf_{\mu}F(\mu)\geq \sum_{i\in[n]}B_i\left(\underline{\mu}-\max_{i,j}\log(d_{ij})\right)-\sum_{i\in[n]}B_i\log\left(\sum_{i\in[n]}B_i\right); \]
    \[ \sup_{\mu}F(\mu)\leq \sum_{i\in[n]}B_i\left(\log\left(\sum_{i\in[n]}B_i\right)-\min_{i,j}\log(d_{ij})\right)-\sum_{i\in[n]}B_i\log\left(\sum_{i\in[n]}B_i\right),\]
    where $ \underline{\mu}\coloneqq \log (\sum_{i\in[n]}B_i/m )-\max_{i,j}\log(d_{ij})+\min_{i,j}\log(d_{ij}).$
\end{Fact}
\begin{proof}[Proof of Fact \ref{fact:F}]
    To begin, we prove that the function $F$ could attain the minimal value at a finite point. It suffices to consider $\mu\in\R^m$ with $\sum_{j\in[m]}e^{\mu_j}=\sum_{i\in[n]}B_i$ due to Fact \ref{fact:homo} and show $F(\mu)> F(\max\{\mu,\underline{\mu}\})$ if $\mu_j<\underline{\mu}$ for some $j\in[m]$. Here, $\max$ is an element-wise operator on $\mu$. If so, then the optimal solution $\mu^*$ with $\sum_{j\in[m]}e^{\mu^*_j}=\sum_{i\in[n]}B_i$ must satisfy $\underline{\mu}\leq\mu^*_j$ for all $j\in[m]$.

        We then show  $F(\mu)> F(\max\{\mu,\underline{\mu}\})$ for all $\mu\in\R^m$ with $\sum_{j\in[m]}e^{\mu_j}=\sum_{i\in[n]}B_i$ and $\mu_j<\underline{\mu}$ for some $j\in[m]$.
  By \eqref{eq:notinJ} and definition of $J_i$, we see that for all $a\in\R$ with 
 $\mu_j<a<\underline{\mu}$, if we increase $\mu_j$ to $a$, it still holds that $ J_i(\max\{\mu,a\})=J_i(\mu)$ and further $h_i(\mu)=h_i(\max\{\mu,a\})$. It follows that  
    \[h_i(\mu)=h_i(\max\{\mu,\underline{\mu}\}) \text{ for all } i\in[n].\]
    This, together with $F=\sum_{i\in[n]}h_i-\ell$, $\mu_j<\max\{\mu_j,\underline{\mu}\}$ for some $j\in[m]$, and the strict monotonic increasing property of $\ell$ w.r.t. each coordinate, implies $F(\mu)> F(\max\{\mu,\underline{\mu}\})$ as desired. We conclude that there is a finite optimal solution $\mu^*$  with $\sum_{j\in[m]}e^{\mu_j}=\sum_{i\in[n]}B_i$ satisfying $\underline{\mu}\leq\mu^*_j$ for all $j\in[m]$.

    Now, using $\underline{\mu}\leq\mu^*_j$, the lower bound for $\inf_{\mu}F(\mu)$ follows from 
    \[\begin{aligned}
    \inf_{\mu}F(\mu)=F(\mu^*)&=\sum_{i\in[n]}B_i\max_{j\in[m]}\{\mu^*_j-\log(d_{ij})\}-\sum_{i\in[n]}B_i\log\left(\sum_{i\in[n]}B_i\right) \\
    &\geq \sum_{i\in[n]}B_i\left(\underline{\mu}-\max_{i,j}\log(d_{ij})\right)-\sum_{i\in[n]}B_i\log\left(\sum_{i\in[n]}B_i\right).
    \end{aligned}
  \]
For the upper bound, it suffices to notice that for all $\mu\in\R^m$ with $\sum_{j\in[m]}e^{\mu_j}=\sum_{i\in[n]}B_i$, we have $\mu_j\leq \log(\sum_{i\in[n]}B_i)$, and hence
\[F(\mu)\leq\sum_{i\in[n]}B_i\left(\log\left(\sum_{i\in[n]}B_i\right)-\min_{i,j}\log(d_{ij})\right)-\sum_{i\in[n]}B_i\log\left(\sum_{i\in[n]}B_i\right). \]
\end{proof}
\subsection{Proof of Lemma \ref{le:dist}}
Noticing $\nabla_j\ell(\mu)=\sum_{i\in[n]}B_ie^{\mu_j}/\sum_{j\in[m]}e^{\mu_j}=q_j(\mu)$, the condition  $|u_j/q_j(\mu)|\leq \epsilon$  is equivalent to 
\[\left|\sum_{i\in[n]} \frac{v_{ij}}{q_j(\mu)}-1\right|\leq \epsilon.\]
    This, together with $p_j=q_j(\mu)$ and $x_{ij}=v_{ij}/p_j$, yields $|\sum_{i\in[n]}x_{ij}-1|\leq \epsilon$. Subsequenctly, we have
    \[1-\epsilon\leq\sum_{i\in[n]}x_{ij}\leq 1+\epsilon\leq \frac{1}{1-\epsilon},\]
    which accords with (A3) of Definition \ref{def:aCE}. To prove (A1), it suffices to combine $p_j=q_j(\mu)$ and $x_{ij}=v_{ij}/p_j$ to give $p^{\top}x_i=\sum_{j\in[m]}v_{ij}=B_i$. 
    
    It is left to check (A2). Indeed, via condition $v_i\in\partial h_i(\mu)$, we can prove an exact version of (A2), i.e., (E2), which is equivalent to \eqref{eq:e2} by Fact \ref{le:e2}.  Consider indices $j\in[m]$ satisfying $x_{ij}>0$. It holds that $v_{ij}>0$ by $x_{ij}=v_{ij}/p_j$, and we further have $j\in J_i(\mu)$ by $v_i\in\partial h_i(\mu)=\conv\{e_j:j\in J_i(\mu)\}$. Note that $J_i(\mu)=\argmax_j\{j\in[m]:p_j/d_{ij}\}$ by \eqref{eq:Jmax}.  We see that \eqref{eq:e2} holds, which completes the proof.

\subsection{Proof of Lemma \ref{le:bound}}
By the optimality of \eqref{eq:SQP}, it holds that
\begin{equation}\label{eq:op}
    \sum_{i\in[n]}v^k_i-\nabla\ell(\mu^k)+\eta(\mu^{k+1}-\mu^k)=0,\ \text{ where }v^k_i\in\partial h_i(\mu^{k+1}). 
\end{equation}
This, together with observation $\sum_{i\in[n]}v_i^{\top}\1=\sum_{i\in[n]}B_i=\nabla\ell(\mu^k)^{\top}\1$, implies that $\1^{\top}\mu^{k+1}=\1^{\top}\mu^k$ for all $k\geq0$. Hence, we see that $\1^{\top}\mu^k\equiv \1^{\top}\mu^0$ for all $k\in[n]$. Given this, to show the boundedness of $\{\mu^k\}_{k\geq0}$, it suffices to show that there is $a\in\R$ such that $\mu^k_j\geq a$ for all $k\geq0$ and $j\in[m]$.

Fact \ref{fact:underlined} and Fact \ref{fact:homo} ensure that for the point $\mu\in\R^m$ and the index $j\in[m]$ that satisfy $\mu_j<\underline{\mu}-\log(\sum_{i\in[n]}B_i/\sum_{i\in[n]}e^{\mu_j})$,  each $v_i\in \partial h_i(\mu)$ satisfies $v_{ij}=0$ for all $i\in[n]$. Notice that by $\1^{\top}\mu^k\equiv \1^{\top}\mu^0$ and the convexity of exponential functions, we have  $\sum_{i\in[n]}e^{\mu^{k+1}_j}\geq ne^{\1^{\top}\mu^0/n}$. We conclude that for all $k\geq0$,
\begin{equation}\label{eq:mua}
    \text{ if } \mu^{k+1}_j<a\coloneqq \underline{\mu}-\log\left(\frac{\sum_{i\in[n]}B_i}{ne^{\1^{\top}\mu^0/n}}\right), \text{ then for all }i\in[n],  \text{ every }v_i\in \partial h_i(\mu^{k+1})\text{ satisfies }  v_{ij}=0.
\end{equation}
Combining \eqref{eq:op} and \eqref{eq:mua}, we see that if $\mu^{k+1}_j<a$, then $\eta(\mu^{k+1}_j-\mu^k_j)=\nabla_j\ell(\mu^k)>0$, and hence $\mu^{k+1}_j>\mu^k_j$. It follows that $\mu^{k+1}_j\geq \min\{\mu^k_j,a\}$. Using $\mu^{k+1}_j\geq \min\{\mu^k_j,a\}$ for $k$ times, we have $\mu^k_j\geq \min\{\mu^0_j,a\}$, which completes the proof. 

\section{Missing Proofs for SGR}\label{appen:sgr}
The following lemmas serve as a basis of the proof of Lemma \ref{le:SP_CE}.
\begin{lemma}\label{le:tech0}
    Given an integer $m>1$ and  a constant $\gamma>2$,  we define the univariate function
    \[\phi:b\mapsto \frac{(1+\sum_{j\in[m-1]} b^{\gamma-1}_j) }{( 1+\sum_{j\in[m-1]} b^{\gamma}_j) }, \quad b\in[0,1].\] It holds that $\phi^*=\sup_{0\leq b\leq1 }\phi(b)\leq \frac{(m-1)^{1/\gamma}}{1-\frac{1.3}{\gamma} }$.
\end{lemma}
\begin{proof}
We first show that the optimal solution $b^*\in\argmax_{b\in[0,1]}\phi(b)$ have equal coordinates, i.e., $b^*_1=b^*_2=\cdots=b^*_{m-1}.$
Since $\phi(b)\leq \phi^*$ holds for all $b\in[0,1]$, we have
\[\tilde{\phi}(b)\coloneqq 1-\phi^*+\sum_{j\in[m-1]}\left( b^{\gamma-1}_j -\phi^*b^{\gamma}_j \right) \leq0 \text{ for all }b\in[0,1]. \]
Notice that $\tilde{\phi}(b^*)=0$ by ${\phi}(b^*)=\phi^*$. This means $b^*$ is an optimal solution of $\tilde{\phi}$, i.e., $b^*\in\argmax_{b\in[0,1]}\tilde{\phi}(b).$
Using the separable structure of $\tilde{\phi}$, we see that $b^*_j$ maximizes $ b^{\gamma-1}_j -\phi^*b^{\gamma}_j$.  Observe that $\phi^*\geq\phi(\frac25\1)>1$. It follows that for all $j\in[m-1]$,
$b^*_j={(\gamma-1})/{(\gamma\phi^*)}\in[0,1].$

Let $\eta=b^*_1$. Then, we have 
\[\phi^*=\phi(b^*)=\frac{1+(m-1)\eta^{\gamma-1}}{1+(m-1)\eta^{\gamma}}. \]
Let $\theta(t)\coloneqq (1+(m-1)t^{\gamma-1})/(1+(m-1)t^{\gamma})$ and $\theta^*=\sup_{t\in[0,1]}\theta(t)$. Clearly, $\theta^*$ is an upper bound for $\phi^*$. We then estimate $\theta^*$. The derivative of $\theta$ is given by
\[\theta^{\prime}(t)= \frac{(m-1)t^{\gamma-2} }{(1+(m-1)t^{\gamma})^2}\left(\gamma-1-\gamma t-(m-1)t^{\gamma}\right) .   \]
Observe that here the fraction is always positive for $t\in[0,1]$, and $\zeta(t)\coloneqq \gamma-1-\gamma t-(m-1)t^{\gamma}$ is strictly monotonically decreasing w.r.t. $t$ on $[0,1]$ with $\zeta(0)=\gamma-1>0$, 
and $\zeta(1)=-m<0$. We know that there is a unique maximizer $t^*\in(0,1)$ for $\theta$ on $[0,1]$ satisfying $\zeta(t^*)=0$.

Let $\bar{t}=(m-1)^{-1/\gamma}(1-\frac{1.3}{\gamma})$. 
Then, using $\log(1+a)\leq a$ for all $a>-1$, we have 
\[(m-1)\bar{t}^{\gamma}=\left(1-\frac{1.3}{\gamma}\right)^{\gamma}=e^{\gamma\log(1-\frac{1.3}{\gamma})}\leq e^{-1.3},\] and further
\begin{equation*}
        \zeta(\bar{t}) \geq  \gamma-1-\gamma(m-1)^{1/\gamma} \bar{t}-(m-1)\bar{t}^{\gamma} \geq 1.3-1-e^{-1.3} >0.
\end{equation*}
This, together with the monotonic decreasing of $\zeta$ and $\zeta(t^*)=0$, implies $t^*\geq \bar{t}$. It follows that
\[\theta^*=\theta(t^*)\leq \frac{\frac1{t^*}+(m-1)(t^*)^{\gamma-1}}{1+(m-1)(t^*)^{\gamma}}= \frac{1}{t^*}\leq \frac{1}{\bar{t}}, \]
where the first inequality uses $t^*\in(0,1)$. Recall that $\theta^*$ is an upper bound for $\phi^*$. We conclude that
$\phi^*\leq \theta^*\leq \frac{(m-1)^{1/\gamma}}{1-\frac{1.3}{\gamma} }$. 
\end{proof}

\begin{lemma}\label{le:tech}
    Consider $a_1,a_2,\ldots,a_m>0$ with $m>1$ and $\gamma>2+\log(m-1)$. The following inequality holds:
    \begin{equation*}
        \frac{\sum_{j\in[m]}a_j^{\gamma-1}}{\sum_{j\in[m]}a_j^{\gamma}}\leq \frac{1}{\max_{j\in[m]}\{a_j\} }\cdot\frac{(m-1)^{1/\gamma}}{1-\frac{1.3}{\gamma} }.
    \end{equation*}
\end{lemma}
\begin{proof}
    Let $\overline{a}=\max_{j\in[m]}\{a_j\}$ and suppose $a_m=\overline{a}$ without loss of generality. Write 
    \[b_j=a_j/\overline{a}\text{ and }\varphi(\gamma)={\sum_{j\in[m]}a_j^{\gamma-1}}/{\sum_{j\in[m]}a_j^{\gamma}}.\] Then, we have 
    \[ \varphi(\gamma)=\frac{1}{\overline{a}}\cdot \frac{1+\sum_{j\in[m-1]} b^{\gamma-1}_j }{1+\sum_{j\in[m-1]} b^{\gamma}_j }\text{ and }b_j\in(0,1]. \]
The desired inequality directly follows from Lemma \ref{le:tech0}.
    
\end{proof}
\subsection{Proof of Lemma \ref{le:SP_CE}}
 By the definition of $v_{ij}$, we see that $\sum_{j\in[m]}v_{ij}=B_i$. This , together with $x_{ij}=v_{ij}/p_j$, implies $p^{\top}x_i=B_i$, which proves (A1) of Definition \ref{def:aCE}. 
    
    Notice that $\nabla_jF_{\delta}(\mu)=\sum_{i\in[n]}v_{ij}-q_j(\mu)$. The condition $|\nabla_j F_{\delta}(\mu)/q_j(\mu)|\leq \epsilon $ implies \[\left|\sum_{i\in[n]}\frac{v_{ij}}{q_j(\mu)}-1\right|\leq\epsilon.\] Using $p_j=q_j(\mu)$ and $x_{ij}=v_{ij}/p_j$ , we further have $|\sum_{i\in[n]}x_{ij}-1|\leq\epsilon$, and hence
    \[ 1-\epsilon\leq \sum_{i\in[n]}x_{ij}\leq 1+\epsilon\leq\frac{1}{1-\epsilon}, \]
    which accords with (A3).
    It is left to verify (A2) of Definition \ref{def:aCE}. Recall that $p^{\top}x_i=B_i$. For all $y_i\in\R^m_+$ satisfying $p^{\top}y_i\geq p^{\top}x_i=B_i$, a simple estimate 
    $d_i^{\top}y_i\geq \frac{B_i}{\max_{j\in[m]}\{\frac{p_j}{d_{ij}}\} }$ holds. 
    Hence, to show (A2), it suffices to show 
    \begin{equation}\label{eq:dijxj}
        (1-\epsilon)d_i^{\top}x_i\leq \frac{B_i}{\max_{j\in[m]}\{\frac{p_j}{d_{ij}}\} } .
    \end{equation}
    We then estimate $d_i^{\top}x_i$:
    \begin{equation}\label{eq:dijxj2}
        \begin{aligned}
            d_i^{\top}x_i = \sum_{j\in[m]} d_{ij}\frac{v_{ij}}{p_j} 
            & = \sum_{j\in[m]} \frac{d_{ij}}{p_j}\cdot \frac{B_ie^{\frac{\mu_j-\log(d_{ij})}{\delta}} }{{\sum}_{j\in[m]} e^{\frac{\mu_j-\log(d_{ij})}{\delta}} } \\
            & = B_i\sum_{j\in[m]} \frac{d_{ij}}{p_j}\cdot  \frac{\left(\frac{e^{\mu_j}}{d_{ij}}\right)^{\frac1{\delta}} }{\sum_{j\in[m]}\left(\frac{e^{\mu_j}}{d_{ij}}\right)^{\frac1{\delta}} }\\
            & =  B_i\sum_{j\in[m]} \frac{d_{ij}}{p_j}\cdot  \frac{\left(\frac{p_j}{d_{ij}}\right)^{\frac1{\delta}} }{\sum_{j\in[m]}\left(\frac{ p_j }{d_{ij}}\right)^{\frac1{\delta}} }\\
            & =  B_i\sum_{j\in[m]}  \frac{\left(\frac{p_j}{d_{ij}}\right)^{\frac1{\delta}-1} }{\sum_{j\in[m]}\left(\frac{ p_j }{d_{ij}}\right)^{\frac1{\delta}} }, \\
        \end{aligned}
    \end{equation}
where the second equality uses the expression of $v_{ij}$ and the forth one uses $p_j={\sum_{i\in[n]}B_i}e^{\mu_j}/{\sum_{j\in[m]} e^{\mu_j}}.$
Further, applying Lemma \ref{le:tech} to \eqref{eq:dijxj2} and noticing the condition $\epsilon\geq(1.3+\log(m-1))\delta$ ensures
\[\epsilon\geq 1.3\delta\cdot e^{-\delta\log(m-1)}+1-e^{-\delta\log(m-1)}=1-(m-1)^{-\delta}(1-1.3\delta) ,\]
we directly have
\[ d_i^{\top}x_i\leq \frac{B_i}{\max_{j\in[m]}\{\frac{p_j}{d_{ij}}\}}\cdot\frac{(m-1)^{\delta}}{1-{1.3}\delta }\leq \frac{B_i}{\max_{j\in[m]}\{\frac{p_j}{d_{ij}}\}}\cdot\frac1{1-\epsilon}.  \]
This accords with \eqref{eq:dijxj}, which completes the proof.
\subsection{Proof of Fact \ref{fact:underlined2}}
    By the expression of $\nabla F_{\delta}$, the inequality $\nabla_{j_0}F_{\delta}(\mu)<0$ is reduced to
    \[\mathlarger{\sum}\limits_{i\in[n]}B_i\left(\frac{e^{\frac{\mu_{j_0}-\log(d_{ij_0})}{\delta}}}{{\sum}_{j\in[m]} e^{\frac{\mu_j-\log(d_{ij})}{\delta}} } \right)<q_{j_0}(\mu)=\mathlarger{\sum}\limits_{i\in[n]}B_i\left(\frac{e^{\mu_{j_0}}}{{\sum}_{j\in[m]}e^{\mu_j}}\right),\]
    which can be further reformulated as
     \[ \mathlarger{\sum}\limits_{i\in[n]}B_i\left(\frac{1}{{\sum}_{j\in[m]} e^{\frac{\mu_j-\mu_{j_0}+\log(d_{ij_0j_0})-\log(d_{ij})}{\delta}} } \right)<q_{j_0}(\mu)=\mathlarger{\sum}\limits_{i\in[n]}B_i\left(\frac{1 }{{\sum}_{j\in[m]}e^{\mu_j-\mu_{j_0}}}\right). \]
    Let $a=\max_{i,j}\{\log(d_{ij})\}-\min_{i,j}\{\log(d_{ij})\}$. It suffices to show  
    \begin{equation}\label{eq:mujmu0}
        \sum_{j\in[m]} e^{\frac{\mu_j-\mu_{j_0}-a}{\delta}}> \sum_{j\in[m]}e^{\mu_j-\mu_{j_0}}.
    \end{equation}
    We partition $[m]$ into two sets:
     \[J_0\coloneqq \{j:\mu_j-\mu_{j_0}\leq a/(1-\delta)+\log(2)\}\quad\text{ and }\quad J_1=\{j:\mu_j-\mu_{j_0}> a/(1-\delta)+\log(2)\}.\]
     We then show $J_1\neq\emptyset$.
     Note that by $\sum_{j\in[m]}q_j(\mu)=\sum_{i\in[n]}B_i$ and $q_j>0$, there exists an indexs $j_1\in[m]$ such that $q_{j_1}(\mu)\geq\sum_{i\in[n]}B_i/m$, i.e., 
     \begin{equation}\label{eq:muj1}
         \mu_{j_1}\geq\log\left(\sum_{j\in[m]}e^{\mu_j}\right)-\log(m). 
     \end{equation}
     On the other hand, by $q_{j_0}(\mu)<e^{\underline{\mu_{\delta}}}$, we have
     \begin{equation}\label{eq:muj0}
     \mu_{j_0}<\log\left(\sum_{j\in[m]}e^{\mu_j}\right)-\log\left(\sum_{i\in[n]}B_i\right)+\underline{\mu_{\delta}}.    
     \end{equation}
     Combining \eqref{eq:muj1} and \eqref{eq:muj0}, using the definition of $\underline{\mu_{\delta}}$, we see that 
     \begin{equation}\label{eq:muj1_muj0}
         \mu_{j_1}-\mu_{j_0}\geq \log\left(\frac{\sum_{i\in[n]}B_i}{m}\right)-\underline{\mu_{\delta}}>\frac{a}{(1-\delta)}+\log(2),
     \end{equation}
     which implies $j_1\in J_1$. Therefore, the index set $J_1$ is non-empty.
     
     The definition of $J_1$, together with $\delta\leq 1/(2+\log(m-1))< \frac12$, yields that for all $j\in J_1$, 
     \[\frac{\mu_{j}-\mu_{j_0}-a}{\delta}+\log\left(\frac12\right)\geq\mu_{j}-\mu_{j_0}.\]
     It follows that for all $j\in J_1$,
     \begin{equation}\label{eq:J1mu}
         \frac12e^{\frac{(\mu_{j}-\mu_{j_0}-a)}{\delta}}>e^{\mu_{j}-\mu_{j_0}}.
     \end{equation}
     We then estimate $ {\sum}_{j\in[m]} e^{\frac{\mu_j-\mu_{j_0}-a}{\delta}}- {\sum}_{j\in[m]}e^{\mu_j-\mu_{j_0}}$ to prove \eqref{eq:mujmu0}.
     \begin{equation*}
         \begin{aligned}
          &\sum\limits_{j\in[m]} e^{\frac{\mu_j-\mu_{j_0}-a}{\delta}}- \sum\limits_{j\in[m]}e^{\mu_j-\mu_{j_0}}\\
             =&\sum_{j\in J_0}\left( e^{\frac{\mu_j-\mu_{j_0}-a}{\delta}}-e^{\mu_j-\mu_{j_0}}   \right)+\sum_{j\in J_1}\left( e^{\frac{\mu_j-\mu_{j_0}-a}{\delta}}-e^{\mu_j-\mu_{j_0}}   \right) \\
             \geq&  \sum_{j\in J_0}-e^{a/(1-\delta)+\log(2)}+\sum_{j\in J_1}\frac12e^{\frac{\mu_{j}-\mu_{j_0}-a}{\delta}} \\
             \geq &-2me^{a/(1-\delta)}+\frac12e^{\frac{\mu_{j_1}-\mu_{j_0}-a}{\delta}}\\
             \geq & -2me^{a/(1-\delta)}+\frac12e^{\frac{-\underline{\mu_{\delta}}+\log(\sum_{i\in[n]}B_i/m)-a}{\delta}}\geq0.
         \end{aligned}
     \end{equation*}
     Here, the first inequality uses $e^{\frac{\mu_j-\mu_{j_0}-a}{\delta}}\geq0$, $\mu_j-\mu_{j_0}\leq a/(1-\delta)+\log(2)$ for $j\in J_0$, and \eqref{eq:J1mu}, the second one uses $j_1\in J_1$, the third one is due to the first inequality of \eqref{eq:muj1_muj0}, and the last one follows from the definition of $\underline{\mu_{\delta}}$. We complete the proof.
\subsection{Proof of Lemma \ref{le:round}}
 We first show that Algorithm \ref{al:round} stops in $m$ steps.  If $J_0=\emptyset$, then the algorithm directly stops. For case $J_0\neq\emptyset$, 
   notice $J_k\subseteq[m]$ and the fact that $J_k\subsetneqq J_{k+1}$ by Line 8. We see that Algorithm \ref{al:round} stops in $m$ steps.

\textit{Proving {\rm (i)}:} The conclusion is trivial when $J_0=\emptyset$. For the case $J_0\neq\emptyset$,  Algorithm \ref{al:round} stops after conducting Line 11, which ensures $\1^{\top}\mu=\1^{\top}\mu^0$.

\textit{Proving {\rm (ii)}:} Notice $F_{\delta}(\mu+t\1)=F_{\delta}(\mu)$ for all $t\in\R$. In Line 11 we have $F_{\delta}(\mu^{out})=F_{\delta}(\mu^{new})$. In Line 1, we show $F(\mu^1)\leq F(\mu^0)$. 
Let $\mu^s$ denote $s\mu^1+(1-s)\mu^0$. By the mean value function theorem and noticing $\mu^1_j=\mu_j^0$ for $j\notin J_{0}$, there is $\bar{s}\in(0,1)$ such that 
\[F(\mu^1)-F(\mu^0)=\nabla F(\mu^{\bar{s}})^{\top}(\mu^1-\mu^0)=\sum_{j\in J_0}\nabla_j F(\mu^s)(\mu^1_j-\mu^0_j).\]
By $\mu^1_j\geq\mu^0_j$ for $j\in J_0$, to show $F(\mu^1)-F(\mu^0)\leq0$, it suffices to show $\nabla_j F(\mu^s)\leq0$. Let $j_0\in\argmax_{j\in J_0}\{\mu^0_j\}$. Then, for $j\in J_0$, it holds that $e^{\mu^1_j}=e^{\mu^1_{j_0}}$ by definition of $\mu^1$ and further
\begin{equation}\label{eq:qmus}
    q_j(\mu^s)\leq b\cdot\frac{e^{\mu^0_{j_0}}}{\sum_{j\in[m]}e^{\mu^s_j} }\leq b\cdot\frac{e^{\mu^0_{j_0}}}{\sum_{j\in[m]}e^{\mu^0_j} } = q_{j_0}(\mu^0)<e^a,\text{ for all }s\in[0,1]
\end{equation}
where the first inequality uses $\mu^s_j\leq \max_{j\in J}\{\mu^0_j\}=\mu^0_{j_0}$ for $j\in J_0$, the second one uses $\mu^s_j\geq\mu^0_j$, and the third one uses the definition of $J_0$ and $j_0\in J_0$. It follows from Fact \ref{fact:underlined2} that $\nabla_jF_{\delta}(\mu)<0$, which proves $F(\mu^1)\leq F(\mu^0)$.

It is left to show: $F_{\delta}(\max\{\mu^k,c_k\})\leq F_{\delta}(\mu^k)$ in Line 7 of each iteration and $F_{\delta}(\max\{\mu^k,{\rm thres}_k\})\leq F_{\delta}(\mu^k)$ in Line 10 of last iteration, where $\max$ is an element-wise operator on $\mu^k$. We only provide a proof for $F_{\delta}(\max\{\mu^k,c_k\})\leq F_{\delta}(\mu^k)$ in Line 7 since the other one uses the same arguments.

To begin, by the definition of $c_k$, we have $\mu^k_j\geq c_k$ for all $j\notin J_{k-1}$. It follows that 
$\max\{\mu^k_j,c_k\}=\mu^k_j$ for $j\notin J_{k-1}$. Moreover, one can verify by induction that in Line 7 of each iteration, 
\begin{equation}\label{eq:induction}
\max_{j\in J_{k-1}}\{\mu^k_j\}\leq \min_{j\notin J_{k-1}}\{\mu^k_j\}=c_k;\quad \mu^k_j \text{ equals for all } j\in J_{k-1}.    
\end{equation}
It follows that $\max\{\mu^k_j,c_k\}=c_k$ for $j\in J_{k-1}$ and $\max\{\mu^k_j,c_k\}=\mu^k_j$ for $j\notin J_{k-1}$ . 
Let $\mu^t=t\max\{\mu^k,c_k\}+(1-t)\mu^k$, where $t\in[0,1]$. Then, for $j\notin J_{k-1}$, we have $\mu^t_j=\mu^k_j$;  for $j\in J_{k-1}$,  $\mu^t_j$ equals and $\mu^t_j\leq c_k$.

Next, we prove that $q_j(\mu^t)\leq e^a$ for $j\in J_{k-1}$ for all $t\in[0,1]$ via the following estimate: 
\begin{equation}\label{eq:qmut}
    \begin{aligned}
    q_j(\mu^t)=b\cdot \frac{e^{\mu^t_j} }{{\sum_{j\in J_{k-1}}e^{\mu^t_j }+\sum_{j\notin J_{k-1}}e^{\mu^t_j} }}&=b\cdot \frac{e^{\mu^t_j}}{{ |J_{k-1}|e^{\mu^t_j}+\sum_{j\notin J_{k-1}}e^{\mu_j}}}\\
    &<b\cdot \frac{e^{{\rm thres}_k}}{{ |J_{k-1}|e^{{\rm thres}_k}+\sum_{j\notin J_{k-1}}e^{\mu_j}}}=e^a, \\
\end{aligned} 
\end{equation}
where the second equality follows from  $\mu^t_j=\mu^k_j$ for $j\notin J_{k-1}$ and $\mu^t_j$ equals for $j\in J_{k-1}$; the inequality uses the strict monotonic increasing property of the function $t\mapsto e^t/(|J_{k-1}|e^t+\kappa)$ ($\kappa>0$) and $\mu^t_j\leq c_k<{\rm thres}_k$ for $j\in J_{k-1}$, which is ensured by definition of $c_{k}$.

Now, by the mean value function theorem, there exists $\bar{t}\in[0,1]$ such that
\[F_{\delta}(\max\{\mu^k,c_k\})-F_{\delta}(\mu^k)=\nabla F_{\delta}\left(\mu^{\bar{t}}\right)^{\top} (\max\{\mu^k,c_k\}-\mu^k)=\sum_{j\in J_{k-1}}\nabla_j F_{\delta}\left(\mu^{\bar{t}}\right)(\max\{\mu^k_j,c_k\}-\mu^k_j).\]
Here, the second equality is due to $\max\{\mu_j,c_k\}=\mu_j$ for $j\notin J_{k-1}$.
To show $F_{\delta}(\max\{\mu^k,c_k\})\leq F_{\delta}(\mu^k)$, by $\mu^k_j\leq c_k$ for $j\in J_{k-1}$, it suffices to show $\nabla_j F_{\delta}(\mu^{\bar{t}})\leq0$ for $j\in J_{k-1}$. This can be seen from \eqref{eq:qmut}, $a\leq \underline{\mu_{\delta}}$, and Fact \ref{fact:underlined2}.

   \textit{Proving {\rm (iii)} :} If $J_0=\emptyset$, then by definition $q_j(\mu^0)\geq e^a$ and $\mu^{out}=\mu^0$. We consider $J_0\neq\emptyset$. We first show that $\mu^k_j\leq {\rm thres}_k$ for $j\in J_{k-1}$. 
   
   To begin, we need to show that  $\mu^1_j\leq {\rm thres}_1$ for $j\in J_0$. Let $s=1$ in \eqref{eq:qmus}. We have $q_j(\mu^1)<e^a$ for $j\in J_0$. This, together with the definition of ${\rm thres}_1$ and  the fact that $\mu^1_j$ equals for $j\in J_0$, implies that for $j\in J_0$,
   \[b\cdot\frac{e^{{\rm thres}_1 }}{|J_0|e^{{\rm thres}_1}+\sum_{j\notin J_0}e^{\mu^1_j}}=e^a>q_j(\mu^1)=b\cdot\frac{e^{\mu^1_j}}{|J_0|e^{\mu^1_j}+\sum_{j\notin J_0}e^{\mu^1_j}}. \]
  This, together with the the strict monotonic increasing of the function $t\mapsto e^t/(|J|e^t+\kappa)$ ($\kappa>0$), implies $\mu^1_j<{\rm thres}_1$, $j\in J_0$. 
  
  On the other hand, recall that \eqref{eq:induction} ensures that in Line 7, $\mu_j^k\leq c_k$ for $j\in J_{k-1}$. It follows from the definition of $\Delta J_k$ that $\mu_j^k\leq c_k<{\rm thres}_k$ for $j\in J_{k-1}$ for $k\geq2$. 
   
   Therefore, we always have $\mu^k_j<{\rm thres}_k$ for $j\in J_{k-1}$. Notice that in Line 10, it holds that $\Delta J_{k}=\emptyset$, which implies $\mu_j^k\geq {\rm thres}_k$ for $j\notin J_{k-1}$. It follows that for $\mu^{new}=\max\{\mu^k,{\rm thres}_k\}$ and $j\in J_{k-1}$, 
\[q_j(\mu^{new})=b\cdot\frac{e^{\max\{\mu^k_j,{\rm thres}_k\}}}{\sum_{j\in J_{k-1}}e^{\max\{\mu_j^k,{\rm thres}_k\}}+\sum_{j\notin J_{k-1}}e^{\mu^k_j}}=b\cdot\frac{e^{{\rm thres}_k}}{\sum_{j\in J_{k-1}}e^{{\rm thres}_k  }+\sum_{j\notin J_{k-1}}e^{\mu^k_j}}=e^a. \]
For $j\notin J_{k-1}$, using $\mu_j^k\geq {\rm thres}_k\geq \mu_{j^{\prime}}^k$ for $j^{\prime}\in J_{k-1}$, we have $q_j(\mu^{new})\geq q_{j^{\prime}}(\mu^k)=e^a$. Hence, in Line 10, one has  $q_j(\mu^{new})\geq e^{a}$ for all $j\in[m]$. By $q_j(\mu+t\1)=q_j(\mu)$, we see that  in Line 11,  $q_j(\mu^{out})\geq e^{a}$ for all $j\in[m]$, which completes the proof.
\subsection{Proof of Theorem \ref{th:SGM}}
     By Lemma \ref{le:SP_CE}, the SGR finds a $\epsilon$-CE in the $k$-th iteration if for all $j\in[m]$,
    \begin{equation}\label{eq:SP_stop}
    \left|\frac{\nabla_j F_{\delta}(\hat{\mu}^k)}{q_j(\hat{\mu}^k)}\right|\leq \epsilon.     
    \end{equation}
    Recall  that $q_j(\hat{\mu}^k)\geq e^{\underline{\mu_{\delta}}}$ by Lemma \ref{le:round}. Using the definition of $\underline{\mu_{\delta}}$ in Fact \ref{fact:underlined2}, the following is a sufficient condition for \eqref{eq:SP_stop}:
    \begin{equation}\label{eq:SP_stop2}
        \|\nabla F_{\delta}(\hat{\mu}^k)\|\leq \epsilon\cdot e^{\underline{\mu_{\delta}}}=\cO\left(\frac{n\epsilon}{m} \right).
    \end{equation}
To obtain \eqref{eq:SP_stop2}, we develop the sufficient descent property for SGR. By Fact \ref{fact:F_epsilon}, the gradient $\nabla F_{\delta}$ is $\sum_{i\in[n]}B_i(1/\epsilon+1)$-Lipschitz continuous. We have
\[F_{\delta}(\mu^{k+1})-F_{\delta}(\hat{\mu}^{k})\leq \nabla F_{\delta}(\hat{\mu}^{k})^{\top}(\mu^{k+1}-\hat{\mu}^{k})+\frac{\sum_{i\in[n]}B_i(1/\delta+1)}{2}\|\mu^{k+1}-\hat{\mu}^{k} \|^2.   \]
Using $\mu^{k+1}-\hat{\mu}^{k}=-\eta   \nabla F_{\delta}(\hat{\mu}^{k})$ and $\eta=\delta/(2\sum_{i\in[n]}B_i(1+\delta))$, we further have
\begin{equation}\label{eq:SP_sufficient}
    F_{\delta}(\mu^{k+1})-F_{\delta}(\hat{\mu}^{k})\leq-\frac{\eta}{2} \| \nabla F_{\delta}(\hat{\mu}^{k})\|^2. 
\end{equation}
On the other hand, by Lemma \ref{le:round}, we have
\begin{equation}\label{eq:SP_descent}
    F_{\delta}(\hat{\mu}^{k+1})\leq F_{\delta}(\mu^{k+1})\quad\text{ for all }k\geq0.
\end{equation}
Summing \eqref{eq:SP_sufficient} and \eqref{eq:SP_descent} over $k=0,1,\ldots,K-1$, we have
\[F_{\delta}(\hat{\mu}^{K})-F_{\delta}(\hat{\mu}^{0})\leq-\frac{\eta}{2}\sum_{k=0}^{K-1}  \| \nabla F_{\delta}(\hat{\mu}^{k})\|^2\leq -\frac{\eta}{2}K\cdot\min_{0\leq k\leq K-1}\{\| \nabla F_{\delta}(\hat{\mu}^{k})\|^2\} . \]
This, together with Fact \ref{fact:F_epsilon} (i) and Fact \ref{fact:F}, implies that
\begin{equation}\label{eq:minFeps}
\begin{aligned}
    \min_{0\leq k\leq K-1}\{\| \nabla F_{\delta}(\hat{\mu}^{k})\|^2\}& \leq 2\frac{\sup_{\mu}F(\mu)-\inf_{\mu}F(\mu)+\delta\log(m)\sum_{i\in[n]}B_i}{\eta K}\\
    &\leq 2\sum_{i\in[n]}B_i\cdot \frac{2\log\left(\frac{\max_{i,j}\{d_{ij}\}}{\min_{i,j}\{d_{ij}\}}\right)+(1+\delta)\log(m)}{\eta K}\\
    & = 2(1+\delta)\left(\sum_{i\in[n]}B_i\right)^2\cdot \frac{2\log\left(\frac{\max_{i,j}\{d_{ij}\}}{\min_{i,j}\{d_{ij}\}}\right)+(1+\delta)\log(m)}{\gamma\delta K},\\
    &\leq 2(1+\delta)n^2\max_{i\in[n]}B_i^2\cdot \frac{2\log\left(\frac{\max_{i,j}\{d_{ij}\}}{\min_{i,j}\{d_{ij}\}}\right)+(1+\delta)\log(m)}{\gamma\delta K},
\end{aligned}
\end{equation}
where the equality uses definition of $\eta$.

Based on \eqref{eq:SP_stop2} and definition of $\delta$, to ensure \eqref{eq:minFeps}, it suffices to let $K=\tilde{\cO}\left(\frac{m^2}{\epsilon^3}\right).$
Recall that the computation of gradient and rounding in each iteration require $\cO(mn)$ and $\cO(m^2)$ time, respectively. The total time complexity is at most $\tilde{\cO}(m^3(m+n)/\epsilon^3)$.

\section{Rounded DCA}\label{appen:round_dca}
In this section, we present the formal algorithm framework of rounded DCA, i.e., Algorithm \ref{al:round_dca}, and its non-asymptotic rate for finding an approximate CE. Since the proof of non-asymptotic rate of rounded DCA is similar to that of SGR, we only provide a brief proof here.
\begin{algorithm}[htbp]
	\caption{Rounded DCA}
	\begin{algorithmic}[1]
		\REQUIRE{Parameters $\epsilon$, $\underline{\mu}$, initial point $\mu^0$, maximum iteration number $k_{max}$, stepsize $\eta$  }
        \STATE Let $k=0$, $a<\underline{\mu}$, $b=\sum_{i\in[n]}B_i$
        \WHILE{ $k\leq k_{max}$ }
        \STATE $\hat{\mu}^k=$Round$(a,\mu^k,b)$
         \IF{ $|\nabla_j g(\hat{\mu}^k)-\nabla_j g(\mu^{k+1})|/q_j(\hat{\mu}^k)<\epsilon$ for all $j\in[m]$ }
          \STATE  \textbf{break}
         \ENDIF
          \STATE Solve \eqref{eq:SQP} via Proposition \ref{pro:subproblem} with $\mu^k$ replaced by $\hat{\mu}^k$
        \STATE $k=k+1$
        \ENDWHILE{ }
		\ENSURE{$\hat{\mu}^k$}
	\end{algorithmic}
	\label{al:round_dca} 
\end{algorithm}

The rounding procedure for DCA is the same as that of SGR except that $\underline{\mu_\epsilon}$ is replaced with $\underline{\mu}\coloneqq \log (\sum_{i\in[n]}B_i/m )-\max_{i,j}\log(d_{ij})+\min_{i,j}\log(d_{ij})$. Using similar arguments of Lemma \ref{le:round}, we can get: (i) $\1^{\top}\hat{\mu}^k=\1^{\top}{\mu}^0$; (ii) $F(\hat{{\mu}}^k)\leq F(\mu^k)$; (iii) $q_j(\hat{\mu}^k)\geq e^a$ for all $j\in[m]$. Then, using Lemma \ref{le:suff} and following the arguments of Theorem \ref{th:SGM}, we have
\begin{equation}\label{eq:round_dca_1}
    \min_{0\leq k\leq K-1}\{\|\hat{\mu}^k-\mu^{k+1}\|^2\}\leq 2\cdot\frac{\sup_{\mu}F(\mu)-\inf_{\mu}F(\mu) }{\eta K}\leq 2\sum_{i\in[n]}B_i\frac{\log\left(\frac{\max_{i,j}\{d_{ij}\}}{\min_{i,j}\{d_{ij}\}}\right)+\log(m) }{\eta K}. 
\end{equation}
Let $k^*\in\argmin_{0\leq k\leq K-1}\{\|\hat{\mu}^k-\mu^{k+1}\|^2\}$. When $K\geq K_0\coloneqq 2\eta\sum_{i\in[n]}B_i/(\log(\frac{\max_{i,j}\{d_{ij}\}}{\min_{i,j}\{d_{ij}\}})+\log(m))$, we have $\|\hat{\mu}^{k^*}-\mu^{k^*+1}\|\leq1$. Then, by Lemma \ref{le:suff}, for $\mu^{k^*+1}=\nabla g(\hat{\mu}^k)-\nabla g(\mu^{k+1})\in \partial F(\mu^{k+1})$,
\begin{equation}\label{eq:round_dca_2}
    |u^{k^*+1}_j|\leq (\eta+e^2q_j(\mu^{k^*+1}))\|\hat{\mu}^{k^*}-\mu^{k^*+1}\|.
\end{equation}
Combining \eqref{eq:round_dca_1} and \eqref{eq:round_dca_2}, we see that for all $j\in[m]$,
\begin{equation}\label{eq:round_dca_3}
    \left|\frac{u^{k^*+1}_j}{q_j(\mu^{k^*+1})} \right|^2\leq 2n\max_{i\in[n]}B_i\left(\frac{\sqrt{\eta}}{q_j(\mu^{k^*+1})}+\frac{e^2}{\sqrt{\eta}}\right)^2\frac{\log\left(\frac{\max_{i,j}\{d_{ij}\}}{\min_{i,j}\{d_{ij}\}}\right)+\log(m) }{K}.
\end{equation}
Let $a=\underline{\mu}-1$. Then, we have $q_j(\hat{\mu}^{k^*})\geq e^{\underline{\mu}-1}$.  By $\|\hat{\mu}^{k^*} - \mu^{k^*+1}\|\leq1$, we further have \[q_j(\mu^{k^*+1})=\frac{\sum_{i\in[n]}B_ie^{\mu^{k^*+1}_j }}{\sum_{j\in[m]}e^{\mu^{k^*+1}_j }}\geq e^{-2}q_j(\hat{\mu}^{k^*})\geq e^{\underline{\mu}-3}=\Theta\left(\frac{n}{m}\right).\]
This, together with \eqref{eq:round_dca_3}, implies that for $\eta=\Theta(\frac{n}{m})$, it suffices to set $K=\tilde{\cO}(\frac{m}{\epsilon^2})$ (by ignoring constant $K_0$) to ensure 
$|{u^{k^*+1}_j}/{q_j(\mu^{k^*+1})}|\leq\epsilon\text{ for all }j\in[m],$ i.e., $\mu^{k^*+1}$ is an $\epsilon$-CE. Hence, the rounded DCA needs $\tilde{\cO}(\frac{m}{\epsilon^2})$ iterations to find an $\epsilon$-CE. We formally state it in following proposition.
\begin{proposition}\label{pro:round_dca}
    Suppose that $a=\underline{\mu}-1$ for the round procedure and $\eta=\Theta(\frac{n}{m})$.
 Then rounded DCA finds an $\epsilon$-CE in at most $\tilde{\cO}(\frac{m}{\epsilon^2})$ iterations. Moreover, the iterates of rounded DCA also converges R-linearly to a CE.
\end{proposition}
\begin{proof}
    It is left to prove the R-linear convergence of rounded DCA. We prove this through the following idea: When the iterates get close to a stationary point, the round procedure would not be performed so that the iterates of rounded DCA act like original DCA iterates, and hence possess R-linear convergence. 
    
To begin, we note that conditions $\1^{\top}\hat{\mu}^k=\1^{\top}{\mu}^0$ and $q_j(\hat{\mu}^k)\geq e^a$ ensure the boundedness of $\{\hat{\mu}^k\}_{k\geq0}$. Let $\mu^*$ be its limiting point. The sufficient descent property ensures that $\|\hat{\mu}^k-\mu^{k+1}\|\to0$. Hence, the point $\mu^*$ is also a limiting point of $\{\mu^{k}\}_{k\geq0}$. Moreover, we have $\dist(0,\partial F(\mu^{k+1})) \to0$ by relative error condition. It follows that $\mu^*$ is a stationary point and then $q_j(\mu^*)\geq e^{\underline{\mu}}$ (for all $j\in[m]$) by Corollary \ref{co:mu}. Therefore, when $\mu^k$ is near $\mu^*$, one has $q_j(\mu^k)\geq e^a$ so that the round procedure is not performed. This means $\hat{\mu}^k=\mu^k$, and that \eqref{eq:suff} and \eqref{eq:relative} hold (for $\mu^k$ near $\mu^*$). Then, \cite[Lemma 2.6]{attouch2013convergence} ensures $\{\mu^k\}_{k\geq0}$ converges to $\mu^*$, and $\mu^k=\hat{\mu}^k$ for sufficiently large $k$. Finally, the $1/2$ K{\L} exponent directly implies the R-linear convergence due to \cite[Theorem 2.3]{schneider2015convergence}.
\end{proof}
\section{Experiment Details and Additional Experiments}\label{appen:additional_exp}
\subsection{Details for Experiments in Sec. \ref{sec:numerical}}
For each instance (with fixed dimension and data type), we repeatedly do the tests $10$ times and calculate the average running time. To ensure a consistent and fair benchmark for algorithm comparison, we standardize the condition number of matrices $D$ and  $B$ (defined as the ratio of the maximum to the minimum) at $100$ through element-wise truncation. 

We use mirror descent to solve the DCA subproblems and adopt constant stepsizes for the SGR. We set parameter $a$ in the rounding procedure small enough so that the rounding is not performed, which does not affect the efficiency of SGR. We heuristically set ${\delta}=\epsilon / 1.3$ for the SGR, which is different from the suggested one in Lemma \ref{le:SP_CE}. Indeed, setting $\delta\leq \epsilon/(1.3+\log(m-1))$ is too conservative and numerical results show that ${\delta}=\epsilon / 1.3$ suffices to ensure the output to be an $\epsilon$-CE; see Sec. \ref{appen:tilde} for the validation.

\begin{table}[htbp]
\centering
\begin{tabular}{>{\bfseries}c*{4}{c}}
\toprule
\multirow{2}{*}{\textbf{Dimension}} & \multicolumn{4}{c}{\textbf{Data Type}} \\
\cmidrule(lr){2-5}
 & Uniform & Log-Normal & Exponential & Integer \\
\midrule
100 & 0.40958 & 0.49373 & 0.41779 & 0.48387 \\
200 & 0.43677 & 0.63498 & 0.49305 & 0.49333 \\
300 & 0.55131 & 0.71102 & 0.54351 & 0.45089 \\
400 & 0.58567 & 0.83008 & 0.6646 & 0.55709 \\
500 & 0.51222 & 0.88317 & 0.69181 & 0.55189 \\
600 & 0.48672 & 0.94475 & 0.66434 & 0.51886 \\
700 & 0.5338 & 0.86891 & 0.72712 & 0.58509 \\
800 & 0.54545 & 1.0289 & 0.76345 & 0.63924 \\
900 & 0.52864 & 0.96698 & 0.77027 & 0.55245 \\
1000 & 0.57917 & 1.0085 & 0.75473 & 0.58537 \\
\bottomrule
\end{tabular}
\caption{Ratio $\frac{{\epsilon}^{\prime}}{\epsilon}$ in Balanced Scenarios}
\label{tb:testing_ratio}
\end{table}
\begin{table}[htbp]
\centering
\begin{tabular}{>{\bfseries}c*{4}{c}}
\toprule
\multirow{2}{*}{\textbf{Dimension}} & \multicolumn{4}{c}{\textbf{Data Type}} \\
\cmidrule(lr){2-5}
 & Uniform & Log-Normal & Exponential & Integer \\
\midrule
500*50 & 0.57196 & 0.91396 & 0.61636 & 0.50733\\
600*50 & 0.48784 & 0.82632 & 0.67539 & 0.51698 \\
700*50 & 0.49921 & 0.84899 & 0.68331 & 0.50150 \\
800*50 & 0.48287 & 0.90268 & 0.62898 & 0.53920 \\
900*50 & 0.52279 & 0.81249 & 0.61100 & 0.48012 \\
1000*50 & 0.53337 & 0.93064 & 0.69852 & 0.44805 \\
\bottomrule
\end{tabular}
\caption{Ratio $\frac{{\epsilon}^{\prime}}{\epsilon}$ in Imbalanced Scenarios}
\label{tb:testing_ratio_imbalance}
\end{table}
\subsection{Validation for Choice of $\delta$}\label{appen:tilde}
We define 
\begin{equation}
   r(\mu)=\max _{i \in[n]}\left\{\frac{\sum_{j \in[m]}\left(\frac{p_j}{d_{i j}}\right)^{\frac{1}{\delta}-1} \max _{j \in [m]}\left\{\frac{p_j}{d_{i j}}\right\}}{\sum_{j \in[m]}\left(\frac{p_j}{d_{i j}}\right)^{\frac{1}{\delta}}}\right\}\text{ with }p_j=\sum_{i\in[n]}B_i\frac{e^{\mu_j}}{\sum_{[j\in[m]}e^{\mu_j} }.
\end{equation}
Recall the proof of Lemma \ref{le:SP_CE}. By \eqref{eq:dijxj} and \eqref{eq:dijxj2}, to ensure the last iterate $\mu^{K}$ to be an $\epsilon$-CE, it suffices to have 
$\epsilon^{\prime}\coloneqq1-\frac1{r(\mu^K)}\leq\epsilon$, or equivalently,
\begin{equation}\label{eq:check}
    \frac{\epsilon^{\prime}}{\epsilon} = \frac{1-\frac{1}{r(\mu^K)}}{\epsilon}\leq1.
\end{equation} 
We then check \eqref{eq:check} in our experiments, where $\delta$ is set to be $\epsilon/1.3$. For $n=m$ case, the ratio $\frac{\epsilon^{\prime}}{\epsilon}$ is reported in Table \ref{tb:testing_ratio} for forty tests. It can be seen that the ratio $\frac{\epsilon^{\prime}}{\epsilon}$ is almost always smaller than $1$ except for a slight excess in two tests.  For the imbalanced dimension case, the ratio $\frac{\epsilon^{\prime}}{\epsilon}$ is reported in Table \ref{tb:testing_ratio_imbalance}, where $\frac{\epsilon^{\prime}}{\epsilon}<1$ always holds. These results validate the rationale of the choice $\delta=\epsilon/1.3$.
%
%
%
%
\end{document}